\documentclass[12pt]{article}

\usepackage[letterpaper, margin=1in]{geometry} %landscape
 
\usepackage{mathrsfs,amsmath,amsfonts,amssymb,latexsym,amsthm,color,
                         graphicx,graphics,float,subfigure,epsf,enumerate,hyperref,algorithm,algorithmic,multirow}
\newtheorem{theorem}{Theorem}[section]
 \newtheorem{corollary}[theorem]{Corollary}
 \newtheorem{lemma}[theorem]{Lemma}
  \newtheorem{assumption}[theorem]{Assumption}
 \newtheorem{proposition}[theorem]{Proposition}

 \theoremstyle{Definition}
 \newtheorem{definition}[theorem]{Definition}
 \theoremstyle{remark}
 \newtheorem{remark}[theorem]{Remark}
   \newtheorem{exercise}{Exercise}
 \numberwithin{equation}{section}

\newcommand{\R}    [1]{\mathbb{R}^{#1}}
\newcommand{\st}      {{\rm{s.t.}}}
\newcommand{\sign}{{\rm{sign}}}

\newcommand{\back}{\begin{acknowledgements}} \newcommand{\eack}{\end{acknowledgements}}
\newcommand{\balg}{\begin{algorithm}}    \newcommand{\ealg}{\end{algorithm}}
\newcommand{\balc}{\begin{algorithmic}}  \newcommand{\ealc}{\end{algorithmic}}
\newcommand{\bali}{\begin{aligned}}      \newcommand{\eali}{\end{aligned}}
\newcommand{\barr}{\begin{array}}        \newcommand{\earr}{\end{array}}
\newcommand{\bass}{\begin{assumption}}   \newcommand{\eass}{\end{assumption}}
\newcommand{\bbma}{\begin{bmatrix}}      \newcommand{\ebma}{\end{bmatrix}}
\newcommand{\bcas}{\begin{cases}}        \newcommand{\ecas}{\end{cases}}
\newcommand{\bcen}{\begin{center}}       \newcommand{\ecen}{\end{center}}
\newcommand{\bcol}{\begin{column}}       \newcommand{\ecol}{\end{column}}
\newcommand{\bcos}{\begin{columns}}      \newcommand{\ecos}{\end{columns}}
\newcommand{\bcon}{\begin{condition}}    \newcommand{\econ}{\end{condition}}
\newcommand{\bcor}{\begin{corollary}}    \newcommand{\ecor}{\end{corollary}}
\newcommand{\bdfn}{\begin{definition}}   \newcommand{\edfn}{\end{definition}}
\newcommand{\benu}{\begin{enumerate}}    \newcommand{\eenu}{\end{enumerate}}
\newcommand{\bequ}{\begin{equation}}     \newcommand{\eequ}{\end{equation}}
\newcommand{\benn}{\begin{equation*}}    \newcommand{\eenn}{\end{equation*}}
\newcommand{\bexa}{\begin{example}}      \newcommand{\eexa}{\end{example}}
\newcommand{\bfig}{\begin{figure}}       \newcommand{\efig}{\end{figure}}
\newcommand{\bfra}{\begin{frame}}        \newcommand{\efra}{\end{frame}}
\newcommand{\bite}{\begin{itemize}}      \newcommand{\eite}{\end{itemize}}
\newcommand{\blem}{\begin{lemma}}        \newcommand{\elem}{\end{lemma}}
\newcommand{\bmat}{\begin{matrix}}       \newcommand{\emat}{\end{matrix}}
\newcommand{\bpma}{\begin{pmatrix}}      \newcommand{\epma}{\end{pmatrix}}
\newcommand{\bpic}{\begin{picture}}      \newcommand{\epic}{\end{picture}}
\newcommand{\bpro}{\begin{proof}}        \newcommand{\epro}{\end{proof}}
\newcommand{\bprp}{\begin{proposition}}  \newcommand{\eprp}{\end{proposition}}
\newcommand{\brem}{\begin{remark}}       \newcommand{\erem}{\end{remark}}
\newcommand{\bsub}{\begin{subequations}} \newcommand{\esub}{\end{subequations}}
\newcommand{\btab}{\begin{table}}        \newcommand{\etab}{\end{table}}
\newcommand{\btar}{\begin{tabular}}      \newcommand{\etar}{\end{tabular}}
\newcommand{\bthe}{\begin{theorem}}      \newcommand{\ethe}{\end{theorem}}
\newcommand{\bvma}{\begin{vmatrix}}      \newcommand{\evma}{\end{vmatrix}}
\newcommand{\bexe}{\begin{exercise}}      \newcommand{\eexe}{\end{exercise}}

\newcommand{\beq}{\begin{equation}}
\newcommand{\eeq}{\end{equation}}
\newcommand{\bequation}{\begin{equation}}
\newcommand{\eequation}{\end{equation}}
\newcommand{\bproof}{\begin{proof}}
\newcommand{\eproof}{\end{proof}}
\newcommand{\benumerate}{\begin{enumerate}}   \newcommand{\eenumerate}{\end{enumerate}}
\newcommand{\bitem}{\begin{itemize}}
\newcommand{\eitem}{\end{itemize}}
\newcommand{\bassumption}{\begin{assumption}}
\newcommand{\eassumption}{\end{assumption}}
\newcommand{\bprop}{\begin{proposition}}
\newcommand{\eprop}{\end{proposition}}
\newcommand{\bal}{\begin{aligned}}
\newcommand{\eal}{\end{aligned}}
\newcommand{\baligned}{\begin{aligned}}
\newcommand{\ealigned}{\end{aligned}}
\newcommand{\bseq}{\begin{subequations}}
\newcommand{\eseq}{\end{subequations}}
\newcommand{\bsubequations}{\begin{subequations}}
\newcommand{\esubequations}{\end{subequations}}
\newcommand{\bcases}{\begin{cases}}
\newcommand{\ecases}{\end{cases}}
\newcommand{\bbmatrix}     {\begin{bmatrix}}
\newcommand{\ebmatrix}     {\end{bmatrix}}
\newcommand{\btheorem}{\begin{theorem}}
\newcommand{\etheorem}{\end{theorem}}

\newcommand{\Ical}{{\cal I}}

\newcommand{\Ucal}{{\cal U}}

\newcommand{\bcls}{\begin{columns}}
\newcommand{\bcl}[1]{\begin{column}{#1\textwidth}}
\newcommand{\ecls}{\end{columns}}
\newcommand{\ecl}{\end{column}}

\usepackage{bm}

\def\bv{{\bm v}}

\def\bx{{\bm x}}

\def\bz{{\bm z}}
\def\bfo{{\boldsymbol 1}}

\def\st{{\rm subject\ to}}
\def\sign{{\rm sign}}

\def\R{{\mathbb R}}
\def\B{{\mathbb B}}

\def\Ical{{\mathcal I}}

\def\Ucal{{\mathcal U}}

\title{A unified approach for projections onto the intersection of $\ell_1$ and $\ell_2$ balls or spheres}

\author{Hongying Liu\thanks{School of Mathematical Sciences, 
               Beihang University,  
               Beijing 100083, China,  
              Tel.: +86-21-82316408,  
              Email: {liuhongying@buaa.edu.cn} }, \and Hao Wang\thanks{School of Information Science and Technology,
                ShanghaiTech University, 
                 Shanghai 201210, China, 
               Tel.:  +86-21-20685389, 
              Email: {wanghao1@shanghaitech.edu.cn}}, \and Mengmeng Song\thanks{
              School of Mathematical Sciences, Beihang University, Beijing 100083, China,  Tel.: +86-21-82316408, 
              Email: {songmengmeng@buaa.edu.cn} }
   }

\begin{document}

\maketitle

\begin{abstract}
This paper focuses on designing a unified approach for computing the
projection   onto the intersection of an $\ell_1$ ball/sphere  and an $\ell_2$ ball/sphere.
% where the ball constraints are often penalized to avoid using  projection-type methods involving  high-complexity projections.
We show that the major computational efforts of solving these problems all
rely on finding the root of the same piecewisely quadratic function, and then propose
a unified numerical method to compute  the root.
In particular, we design breakpoint search methods with/without sorting incorporated with bisection, secant and Newton methods to
 find the interval
containing the root, on which the root has a closed form.
It can be shown that our proposed algorithms without sorting possess $O(n \log n)$ worst-case complexity and $O(n)$ in practice. The efficiency of our proposed algorithms are demonstrated in  numerical experiments.


 % \PACS{PACS code1 ,  PACS code2 ,  more}
% \PACS{PACS code1 ,  PACS code2 ,  more}
% \subclass{MSC code1 ,  MSC code2 ,  more}
\end{abstract}

%\subclass{49J53 ,   49K99 ,  more}

\section{Introduction}

In this paper, we consider designing  a unified  numerical method for computing  the solution of the following three types of problems:
 projection onto the intersection of an $\ell_1$ ball and an $\ell_2$ ball:

\begin{equation}\label{prob.b1b2.0}
%\begin{array}{ll}
\mathop{\rm minimize}\limits_{\bx\in \R^n}   \tfrac{1}{2}\|\bx-\bv\|_2^2\quad \st \ \bx\in \mathbb{B}_1^t \cap \mathbb{B}_2,
%\st  &\ \ \|\bx\|_1\le t\\
%     & \tfrac{1}{2} \|\bx\|_2^2 \le \tfrac{1}{2}.
%\end{array}
\end{equation}
 projection onto the intersection of an $\ell_1$ ball and an $\ell_2$ sphere:

\begin{equation}\label{prob.b1s2.0}
\mathop{\rm minimize}\limits_{\bx\in \R^n}   \tfrac{1}{2}\|\bx-\bv\|_2^2\quad \st \  \bx\in\mathbb{B}_1^t\cap\mathbb{S}_2,
\end{equation}
and 
 projection onto the intersection of an $\ell_1$ sphere and  an $\ell_2$ sphere:
\begin{equation}\label{prob.s1s2.0}
\mathop{\rm minimize}\limits_{\bx\in \R^n}   \tfrac{1}{2}\|\bx-\bv\|_2^2\quad \st \  \bx\in\mathbb{S}_1^t\cap\mathbb{S}_2.
\end{equation}
%and  minimizing a linear function on the intersection of an $\ell_1$ sphere and an $\ell_2$ sphere.
%\begin{equation}\label{prob.l.b1b2.0}
%\mathop{\rm minimize}\limits_{\bx\in \R^n}   -\bv^T\bx\quad \st \  \bx\in\mathbb{B}_1^t\cap\mathbb{B}_2.
%\end{equation}
Here the $\ell_2$ (i.e., Euclidean) norm on   $\R^n$ is indicated as $\|\cdot \|_2$ with the unit $\ell_2$
ball (sphere) defined as $\B_2 := \{ \bx\in \R^n: \|\bx\|_2 \le 1\}$   $(\mathbb{S}_2=\{ \bx\in \R^n: \|\bx\|_2 = 1\}$), and the $\ell_1$ norm is indicated as $\|\cdot\|_1$ with the $\ell_1$
ball  (sphere)  with radius $t$ denoted as $\B_1^t:=\{\bx\in\R^n:\|\bx\|_1\le t\}$  ($\mathbb{S}_1^t=\{\bx\in \R^n: \|\bx\|_1=t\}$).
Notice that $\|x\|_2\le \|x\|_1 \le \sqrt{n}\|x\|_2$.  
Trivial cases for the problems of interests are: 
(a) $t\le1$, in this case $\|x\|_1 \le t$ implies $\|x\|_2 < 1$, meaning $\B_1^t \subset \B_2$, 
 $\mathbb{S}_1^t \cap \B_2 = \emptyset$ and $\mathbb{S}_1^t \cap \mathbb{S}_2 = \emptyset$. 
(b)  $t \ge \sqrt{n}$, in this case $\|x\|_2\le 1$ implies $\|x\|_1 < t$, meaning $\B_2 \subset \B_1^t$, 
 $\mathbb{S}_1^t \cap \B_2 = \emptyset$ and $\mathbb{S}_1^t \cap\mathbb{S}_2 = \emptyset$.   Without loss of generality, we assume
$1\le t \le \sqrt{n}$ in the remainder of this paper.

Problems \eqref{prob.b1b2.0} \eqref{prob.b1s2.0} and \eqref{prob.s1s2.0}  
  arise widely in modern science, engineering and business.  For example, the gradient projection
  methods for  Sparse Principal Component Analysis (sPCA)\cite{Jolliffe03,Luss11,Sigg08,Witten09}  often
 involve problems of \eqref{prob.b1b2.0} or \eqref{prob.s1s2.0}, and \eqref{prob.s1s2.0} is  also an integral part in efficient sparse non-negative matrix factorization \cite{Hoyer04,Potluru13}, supervised online autoencoder intended for classification using neural networks that features sparse
activity and sparse connectivity \cite{Thom15},  and  dictionary learning with sparseness-enforcing projections \cite{Thom15}.
 Problem  
  \eqref{prob.b1s2.0} often arises in Sparse Generalized Canonical Correlation Analysis (SGCCA)\cite{Tenenhaus14}, and 
 Witten et al. \cite{Witten09} use \eqref{prob.b1s2.0}  
 for computing  the rank-1
approximation for a given matrix along with a block coordinate decent method, which can be applied to sparse principal components and canonical correlation.

 Our contribution in this paper can be summarized as follows:

\begin{itemize}
\item We propose a unified analysis for solving these problems. Specifically, we show that
their solutions can all be determined  by the root of a  piecewisely quadratic  auxiliary function.

\item A series of  properties of the proposed auxiliary function are provided, which provide
detailed characterization of the solutions of these problems.

\item A unified method with/without sorting is designed for finding the root of the auxiliary function, which accounts for
the major computational efforts of solving these problems.
%Specifically,  we design breakpoint search methods
%with/without sorting incorporated with bisection, secant and Newton methods, which are shown to be efficient by
%experiments.
\end{itemize}

\subsection{Organization}
In the remainder of this section, we outline our notation and introduce various concepts that will be employed throughout the paper.
In \S\ref{sec.existing}, we discuss the  most related existing problems and algorithms.
In \S\ref{sec.function}, we introduce our proposed auxiliary function  and provide a series of properties of the auxiliary function.
We use the proposed auxiliary function to characterize the optimal solutions in \S\ref{sec.solution}.
A unified  algorithm is proposed in \S\ref{sec.algorithm} for finding the root  of the auxiliary function.
 The results of numerical results  are shown  in \S\ref{sec.experiment}.  Concluding remarks are provided   in \S\ref{sec.conclusion}.

\subsection{Notation}

 For any  $\bx\in \R^n$, let  $x_i$ be the $i$-th element of $\bx$  and $\R^n_+$ be the nonnegative orthant of $\R^n$, i.e.,
$\R^n_+ := \{ \bx\in\R^n: x_i \ge 0, i=1,\ldots, n\}$.
Denote the  soft thresholding operator in $\R^n$ with threshold $\lambda>0$ by $S_\lambda(\cdot)$, i.e.,  for any $\bx\in\R^n$,
$(S_\lambda(\bx))_i = \sign(x_i)\max(|x_i|-\lambda)$
%$$(S_\lambda(\bx))_i=\begin{cases}
%x_i-\lambda, &\text{if } x_i\ge \lambda\\
%0,       &\text{if }  |x_i|<\lambda \\
%x_i+\lambda,  &\text{if }  x_i\le -\lambda
%\end{cases}
%$$
for $i=1,\ldots,n$.
Given $\bv\in\mathbb{R}^n$, denote $\bv^+$ as the projection of $\bv$ onto the nonnegative orthant $\mathbb{R}^n_+$,
i.e. $\bv^+=\max(\bv, 0)$.
The $\ell_{p,q}$ norm  of $\bx$  is defined as
$\|\bx\|_{p,q} =  \left(  \sum_{i=1}^G \|\bx_i\|_q^p  \right)^{1/p}$ with $\bx_i\in\mathbb{R}^{n_i}$ and $\sum_{i=1}^G n_i = n$,
where $G$ is the number of groups.   For a compact set $\Omega\subset\mathbb{R}^n$ and $\bv\in\mathbb{R}^n$,
denote 
$\text{proj}_{\Omega}(\bv) = \arg\min_{\bx\in\mathbb{R}^n}\|\bv - \bx\|_2^2$.
The function $\phi: \mathbb{R}^n\to \bar{\mathbb{R}}:=\mathbb{R}\cup\{+\infty\}$ is convex, then the subdifferential of $\phi$ at
$\bar \bx$ is given by
$\partial \phi(\bar \bx) : = \{  \bz\in \R^n \mid \phi(\bar \bx) + \bz^T(\bx-\bar \bx) \le \phi(\bx), \forall \bx\in\mathbb{R}^n\}.$

Denote ${\bf 1}\in \R^n$ be the vector of all ones.  The largest  and the  second-largest  of $\bv$ are denoted as $v_{\max} = \max(v_1, ..., v_n)$ and
$v_{\text{2nd-max}}$, respectively.
To simplify the analysis, we assume $\lambda_j, j=1,\ldots,k$ are  the  $k$   distinct
components of $\bv$ such that $\lambda_1 >\ldots > \lambda_k$ with $\lambda_1 = v_{\max}$ and $\lambda_{k+1}=-\infty$.

\section{Related methods}\label{sec.existing}
%  As for projection onto a single $\ell_1$ ball, many  algorithms    have appeared.
%  For example,
%%The projection onto a $\ell_2$ ball is simply a rescaling of the given vector.
%    $\ell_1$ ball projection algorithms  have emerged
%    \cite{Duchi08,Liu09,Songsiri11}
%    %\cite{Berg08,Condat16,Duchi08,Liu09,Songsiri11}
%    and the projection onto $\ell_{1,2}$ ball is considered in \cite{Gong11}.
%     However, when it comes to the projection onto the intersection of the $\ell_1$ and $\ell_2$ balls or spheres,   not much work has appeared in the literature.
We discuss the most related works in this section.

{\bf Projection onto $\ell_1$ ball}.
  As for projection onto a single $\ell_1$ ball, many  algorithms    have emerged. It can be shown 
 %The projection onto a $\ell_2$ ball is simply a rescaling of the given vector.
   \cite{Duchi08,Liu09,Songsiri11} that
    % \cite{Berg08,Condat16,Duchi08,Liu09,Songsiri11}
     the projection of $\bv$ onto $\mathbb{B}_1^t$ can be characterized by
    the root of the auxiliary function
\[ \psi(\lambda)  :=  \sum_{i=1}^n \max(v_i-\lambda,0) -t = \| (\bv - \lambda \boldsymbol{1})^+\|_1 - t.\]
The properties   of $\psi$ are summarized  in the proposition below.

\begin{proposition}\label{lem.phi}
Function $\psi$ is continuous, strictly decreasing and piecewisely linear on $(-\infty,v_{\max}]$ with
breakpoints $\lambda_1, \ldots, \lambda_k$, and $\psi\equiv -t < 0$ for any $\lambda \ge v_{\max}$.
\end{proposition}
%\begin{proof} We can write $\psi$ on  $(\lambda_{j+1} ,  \lambda_j]$ as
%\[ \psi(\lambda) = s_j - I_j\lambda - t, \  j=1,...,k\]
%by \eqref{eq:phi} and \eqref{eq:index}.  Therefore $\psi$ is continuous, strictly decreasing and piecewisely linear on $(-\infty,v_{\max}] $.
% In particular, $\psi(\lambda)=-t<0$ for any $\lambda >  \lambda_1$.
%\end{proof}
%
By Proposition \ref{lem.phi}, $\psi(\lambda)=0$ has a unique root on $(-\infty,v_{\max}) $ since $\psi(\lambda)\to + \infty$ as $\lambda\to -\infty$ and $\psi(v_{\max})=-t<0$.  The algorithms for computing the $\ell_1$ ball projection are summarized and compared  in \cite{Condat16}, in which
an efficient algorithm is also proposed with worst-case complexity $O(n^2)$ and observed complexity $O(n)$.

{\bf Group ball projection}.
The first  related work is the Euclidean projection onto the intersection of $\ell_1$ and $\ell_{1,q}$ norm balls ($q=2$ or $q=\infty$) proposed by Su et al. \cite{LiFF12}.  With $q=2$ and one group, this problem reverts to \eqref{prob.b1b2.0}.
  They proved that the projection can be reduced to finding the root of an auxiliary function% , which can be reduced  to
  %Suppose  $\bv\in\mathbb{R}^n$ is the  given vector to be projected.
%This problem can be formulated as
%\begin{equation}
%\begin{array}{ll}
%\mathop{\rm minimize}\limits_{\bx\in \R^n} & \|\bx-\bv\|_2^2\\
%\st  &\ \ \|\bx\|_1\le t\\
%     & \tfrac{1}{2} \|\bx\|_2^2 \le \tfrac{1}{2}.
%\end{array}
%\end{equation}
 \[ \phi_1(\lambda) := \|S_\lambda(\bv)\|_1/\|S_\lambda(\bv)\|_2 -t,\quad \lambda\in [0, v_{\max}).\]
% for \eqref{prob.b1b2.0}.
Su et al. \cite{LiFF12} studied the properties of this  auxiliary function, which  are  summarized
 in the following Lemma \ref{lem.LiFF}. Based on this lemma, a bisection algorithm is proposed to find the root of $\phi_1$.

\begin{lemma}[\cite{LiFF12}{ Theorem 1}]\label{lem.LiFF} The following statements hold true:
(i) $\phi_1$ is  continuous piece-wise smooth   on $(0, v_{\max})$;
(ii) $\phi_1$ is monotonically decreasing and  has a unique root in $(0, v_{\max})$.
\end{lemma}

\noindent{\bf Remark:}
However, part (ii) of this lemma may not  hold in general.  We show this by the following two counterexamples.

 {\em Example 1.} Consider $n=2, t=1.2$ and $\bv=(1,0)$.  Then for $\lambda\in (0,1)$
\[\phi_1(\lambda)=\frac{\max(1-\lambda, 0)+\max(-\lambda, 0)}{\sqrt{(\max(1-\lambda, 0))^2+(\max(-\lambda, 0))^2}}-1.2=-0.2.\]
Obviously, for this instance,   $\phi_1$ has no root on $(0,1)$. Therefore, Lemma~\ref{lem.LiFF} does not hold.

{\em Example 2.} Consider $n=3, t=\sqrt{2}$ and $\bv=(1,1,0)$.  Then for $\lambda\in (0,1)$
\[\phi_1(\lambda)=\frac{2\max(1-\lambda, 0)+\max(-\lambda, 0)}{\sqrt{2(\max(1-\lambda, 0))^2+(\max(-\lambda, 0))^2}}-\sqrt{2}\equiv 0.\]
Clearly, any point in $(0,1)$ is the root of $\phi_1$, so that Lemma~\ref{lem.LiFF} does not hold.

{\bf Sparseness-enforcing projection operator}.
Another related work is the ``sparseness-enforcing projection operator''
 proposed by Hoyer \cite{Hoyer04}, which requires the solution to satisfy
 a normalized smooth ``sparseness measure'' defined by
\[\sigma:\R^n\setminus \{0\}\to [0,1],\ \bv\mapsto (\sqrt{n}-\|\bv\|_1/\|\bv\|_2)/(\sqrt{n}-1).\]
This leads to   solving the problem of \eqref{prob.s1s2.0}.

Theis et al. \cite{Fabian05} shown that the projection  is almost surely unique for $\bv$ drawn from a continuous distribution, and 
if it is unique,  the projection is shown to be     determined by the root of $\phi_1$.
We summarized the results in Lemma \ref{lem.Thom15}. Algorithms for solving \eqref{prob.s1s2.0} mainly include
the alternating projection  method  in \cite{Hoyer04,Fabian05}, the method of Lagrange multipliers based on sorted $\bv$ in \cite{Potluru13},
and the method in \cite{Thom15} based on computing the root of the auxiliary function $\phi_1(\lambda)$.

\begin{lemma}\label{lem.Thom15}(\cite[Lemma 3 in Appendix]{Thom15}) Let $\bv\in\mathbb{R}^n_+\setminus {\mathbb{R}^n_+\cap\mathbb{S}^t_1\cap\mathbb{S}_2}$ be a point such that ${\rm proj}_{\mathbb{R}^n_+\cap\mathbb{S}^t_1\cap\mathbb{S}_2}(\bv)$ is unique and $\sigma(\bv)<\sigma^*$. Then   $\phi_1$ is well defined and the following hold:
\begin{enumerate}
\item[(i)] $\phi_1$ is continuous on $(0, v_{\max})$;
\item[(ii)] $\phi_1$ is differentiable on $(0, v_{\max})\setminus\{v_1, \ldots, v_n\}$.
\item[(iii)] $\phi_1$ is strictly  decreasing on $(0, v_{\rm 2nd-max})$, and is constant  on $[v_{\rm 2nd-max}, v_{\max})$.
\item[(iv)] $\phi_1$ has   a unique root $\alpha^*\in (0, v_{\rm 2nd-max})$, and   ${\rm proj}_{\mathbb{R}^n_+\cap\mathbb{S}^t_1\cap\mathbb{S}_2}(\bv)=\frac{S_{\alpha^*}(\bv)}{\|S_{\alpha^*}(\bv)\|_2}$.
\end{enumerate}
\end{lemma}

\noindent{\bf Remark:} Here the condition $\sigma(\bv)<\sigma^*$ holds if and only $\|\bv^+\|_1>t\|\bv^+\|_2$. Compared with Theorem \ref{thm.nonneSS}, Lemma \ref{lem.Thom15} may not include the situation where the projection is not unique or the projection is unique but $\sigma(\bv)\ge \sigma^*$.

{\bf Projection onto intersection of  an $\ell_1$ ball and  an $\ell_2$ sphere.}
Tenenhaus et al. \cite{Tenenhaus14} provided a close form of the solution \eqref{prob.b1s2.0}. 
The algorithms for solving \eqref{prob.b1s2.0} mainly include the root finding with bisection proposed by \cite{Tenenhaus14} and the root finding
method with sorting $\bv$ by \cite{Gloaguen17}.
%the projection onto the nonconvex set $\B_1^t\cap \mathbb{S}_2$, i.e.
%\begin{equation}\label{prob.b1s2.0}
%\mathop{\rm minimize}\limits_{\bx\in \R^n}  \tfrac{1}{2} \|\bx-\bv\|_2^2  \quad \st \  \bx\in\mathbb{B}_1^t\cap\mathbb{S}_2
%\end{equation}
%which is equivalent to \eqref{prob.s1s2.0}.
%  \eqref{prob.b1s2.0} arises in Sparse Generalized Canonical Correlation Analysis (SGCCA), which is attacked by
% This method
% was based on block relaxation to maximize a
%cost function, and at each block relaxation substep, an subproblem similar to \eqref{prob.b1s2.0}
%needs to be solved.
Let $\tilde \bv = |\bv| $ and suppose the elements are sorted in descent order.
They analyzed the properties of  $\phi_1(\lambda)$ in the following lemma.

\begin{lemma}\label{lem.Gloaguen}(\cite[Proposition 1]{Gloaguen17}) The following statements hold true.
\begin{itemize}
\item[(i)] $\phi_1$ is continuous and decreasing.
\item[(ii)] Let $n_{\max}$ be the number of elements of $\bv$ equal to $v_{\max}$. For $t\in [\sqrt{n_{\max}}, \sqrt n]$, there exists $i\in \{1,\ldots, n\}$ and $\delta\in [0, \tilde v_i-\tilde v_{i+1})$ such that $\phi_1(\tilde v_i-\delta)=0$.
\item[(iii)] $\delta$ is a solution of a second degree polynomial equation.
\end{itemize}
\end{lemma}
%
%For computing the root of $\phi_1$,  the algorithm proposed in \cite{Gloaguen17} aims to finding $i$ such $\phi_1(\tilde v_{i+1})>0$ and $\phi_1(\tilde v_{i})<0$ and then obtain $\delta$ accordingly. After that, they choose $\lambda = \tilde v_i-\delta$ such that $\phi_1(\lambda)=0$.\\

\noindent{\bf Remark.} Part  (ii) of Lemma \ref{lem.Gloaguen} shown that $n_{\max}\le t^2$ is the sufficient condition that $\phi_1(\lambda)=0$ have a root on $(0, \tilde v_{\max})$. However, Example 1 is a counterexample indicating  that  $n_{\max}\le t^2$
 is not   sufficient to guarantee  $\phi_1(\lambda)$ has a root on $(0,   v_{\max})$.

\section{Proposed auxiliary function}\label{sec.function}

Based on the discussion in \S\ref{sec.existing}, most existing projection algorithms onto the intersection of  $\ell_1$ and $\ell_2$ balls/spheres
 are constructed by using the auxiliary function $\phi_1(\lambda)$.  Our proposed methods are based on different
 auxiliary functions for characterizing the properties of the projections, which is the main focus of this section.

We first  show that  the solutions of  \eqref{prob.b1b2.0}/\eqref{prob.b1s2.0}/\eqref{prob.s1s2.0},
have  the same sign as the given $\bv$, %  onto $\B_1^t \cap \B_2$, which is to solve \eqref{prob.b1b2.0}.
% The follow proposition shows that every nonzero component of the projection has the same sign of $\bv$,
which is a generalized result of the $\ell_1$ ball
projection in \cite{Duchi08,Gong11}.

\begin{proposition}\label{prop.sign}
Let   $\bx$ be the first-order  optimal  to   \eqref{prob.b1b2.0}/\eqref{prob.b1s2.0}/\eqref{prob.s1s2.0}, then $v_ix_i \ge 0$ for $i=1,\ldots,n$.
\end{proposition}
\begin{proof}Assume  by contradiction that there exists $i_0$ such that $v_{i_0}x_{i_0} < 0$.  Define $\hat\bx$ such that
$\hat x_{i_0} = - x_{i_0}$ and $\hat x_i = x_i$ for all $i\neq i_0$, implying $\|\hat\bx\|_p = \|\bx\|_p, p=1,2$.
Therefore,  $\hat\bx$ is feasible for \eqref{prob.b1b2.0}/\eqref{prob.b1s2.0}/\eqref{prob.s1s2.0}.  However,
\[ \tfrac{1}{2}(\|\bx-\bv\|_2^2 - \|\hat\bx-\bv\|_2^2) =\tfrac{1}{2}((x_{i_0} - v_{i_0})^2 - (-x_{i_0} - v_{i_0})^2) = - 2v_{i_0}x_{i_0} > 0,\]
%and 
%\[ -\bv^T\bx+\bv^T\hat\bx=- 2v_{i_0}x_{i_0} > 0,\]
contradicting that $\bx$ is optimal for \eqref{prob.b1b2.0}/\eqref{prob.b1s2.0}/\eqref{prob.s1s2.0}.  This completes the proof.
\end{proof}

Using the  symmetry of the feasible region stated in Proposition~\ref{prop.sign}, we can transform  the original problems
 \eqref{prob.b1b2.0}, \eqref{prob.b1s2.0} and \eqref{prob.s1s2.0}  to
their  corresponding problems   restricted in $\mathbb{R}^n_+$, so that from now on
% For simplicity and without loss of generality, we can consider our projection problems to cases
%where $\bv\in\mathbb{R}^n$ and restrict the solution $\bx$  to $\mathbb{R}^n_+$.
% From now on,
we can focus on  the following problems
% For the projection onto $\B_1^t \cap \B_2$, we can instead consider the
%  projection onto  $\B_1^t \cap \B_2 \cap \R^n_+$:
%\begin{equation}\label{prob.b1b2}
%\begin{array}{ll}
%\mathop{\rm minimizing}\limits_{\bx \ge0 } & \tfrac{1}{2}\|\bx-\bv\|_2^2\\
%\st  &\quad \sum_{i=1}^nx_i\le t\\
%     & \tfrac{1}{2} \sum_{i=1}^nx_i^2\le \tfrac{1}{2}.
%\end{array}
%\end{equation}
%For the projection onto $\mathbb{S}_1^t \cap \mathbb{S}_2$,
%it can be transformed to the projection onto $\mathbb{S}_1^t \cap \mathbb{S}_2 \cap \R^n_+$
%\begin{equation}\label\eqref{prob.s1s2}
%\begin{array}{ll}
%\mathop{\rm minimizing}\limits_{\bx \ge0 } & \tfrac{1}{2}\|\bx-\bv\|_2^2\\
%\st  &\quad \sum_{i=1}^nx_i= t\\
%     & \tfrac{1}{2} \sum_{i=1}^nx_i^2= \tfrac{1}{2}.
%\end{array}
%\end{equation}
%Finally, for $\mathbb{B}_1^t \cap \mathbb{S}_2$, we consider
%\begin{equation}\label\eqref{prob.b1s2}
%\begin{array}{ll}
%\mathop{\rm minimizing}\limits_{\bx \ge 0 } & \tfrac{1}{2}\|\bx-\bv\|_2^2\\
%\st  &\quad \sum_{i=1}^nx_i\le t\\
%     & \tfrac{1}{2} \sum_{i=1}^nx_i^2= \tfrac{1}{2},
%\end{array}
%\end{equation}
%We also consider maximizing a linear function $\mathbb{B}_1^t \cap \mathbb{B}_2$ in the nonnegative orthant, i.e., $\mathbb{B}_1^t \cap \mathbb{B}_2 \cap \R^n_+$, which can be formulated as
%\begin{equation}\label{prob.l.b1b2}
%\begin{array}{ll}
%\mathop{\rm minimize}\limits_{\bx \ge 0 } & -\bv^T\bx \\
%\st  &\quad \sum_{i=1}^nx_i\le t\\
%     & \tfrac{1}{2} \sum_{i=1}^nx_i^2 \le \tfrac{1}{2}.
%\end{array}
%\end{equation}
%This problem corresponds to \eqref{prob.s1s2.0}.
\begin{equation}\label{prob.b1b2}
\mathop{\rm minimize}\limits_{\bx}   \tfrac{1}{2}\|\bx-\bv\|_2^2\quad \st \ \bx\in  \mathbb{B}_1^t \cap \mathbb{B}_2 \cap \R^n_+,
\end{equation}
\begin{equation}\label{prob.b1s2}
\mathop{\rm minimize}\limits_{\bx}   \tfrac{1}{2}\|\bx-\bv\|_2^2\quad \st \  \bx\in   \mathbb{B}_1^t\cap\mathbb{S}_2 \cap \R^n_+,
\end{equation}
\begin{equation}\label{prob.s1s2}
\mathop{\rm minimize}\limits_{\bx}   \tfrac{1}{2}\|\bx-\bv\|_2^2\quad \st \  \bx\in   \mathbb{S}_1^t\cap\mathbb{S}_2\cap \R^n_+,
\end{equation}
% \begin{equation}\label{prob.l.b1b2}
%\mathop{\rm minimize}\limits_{\bx}   -\bv^T\bx\quad \st \  \bx\in   \mathbb{B}_1^t\cap\mathbb{B}_2 \cap \R^n_+,
%\end{equation}
corresponding to \eqref{prob.b1b2.0}, \eqref{prob.b1s2.0} and \eqref{prob.s1s2.0},  % and \eqref{prob.l.b1b2.0}
 respectively.

% \subsection{Proposed auxiliary function}

 %We provide a series of results  to characterize the properties of the solution of \eqref{prob.b1b2} and \eqref{prob.b1b2}.

%
% In this section, we define and analyze two auxiliary functions to characterize the properties of the projections.
%

We define the following   univariate function
for given $\bv\in\mathbb{R}^n$ and  $t>0$:
\[
\begin{aligned}
 % \quad \phi(\lambda) := &\ \left(\sum_{i=1}^n\max(v_i - \lambda,0)\right)^2 - t^2\sum_{i=1}^n (\max(v_i-\lambda,0))^2 \\
   \phi(\lambda) : = &\ \| (\bv - \lambda \boldsymbol{1})^+\|_1^2 - t^2 \| (\bv - \lambda \boldsymbol{1})^+\|_2^2.
\end{aligned}
\]
Denote the index set of components greater than or equal to a given $\lambda$:
$$\Ical_\lambda = \{i : v_i \ge \lambda, i=1,\ldots,n \}\ \text{ and } \   I_\lambda =|\Ical_\lambda|.$$
The summations of those components and the squared components are denoted  as
$s_\lambda  = \sum\limits_{i\in \Ical_\lambda} v_i$ and  $w_\lambda  = \sum\limits_{i\in \Ical_\lambda} v_i^2,$
respectively.
For simplicity,  for the $k$ distinct values in $\bv$, we write
\[ \Ical_j = \Ical_{\lambda_j},\  I_j = I_{\lambda_j}, \ s_j= s_{\lambda_j}, \ w_j = w_{\lambda_j},\ \text{ for }  j=1, \ldots k.\]
Notice that since $\lambda_j > \lambda_{j+1}$, 
\[ \Ical_j \subset \Ical_{j+1},\ I_j < I_{j+1},\ s_j < s_{j+1},\ w_j < w_{j+1}, \ \text{ for } j=1,\ldots, k-1.\]
%Denote $\Ical_{\lambda_j}=\Ical_j, I_{\lambda_j}= I_j, s_{\lambda_j} = s_j, q_{\lambda_j} = \varphi_j, j=1, \ldots k.$
In particular, it is obvious that
\begin{equation}\label{eq:index}
\begin{cases}
\Ical_\lambda = \Ical_j, I_\lambda = I_j, s_\lambda = s_j, w_\lambda = w_j, & \    \forall  \lambda\in (\lambda_{j+1}, \lambda_{j}],\  j=1,...,k,\\
I_k=n, s_k=\sum\limits_{i=1}^nv_i, w_k=\sum\limits_{i=1}^nv_i^2. &
\end{cases}
\end{equation}
Therefore, we can rewrite % $\psi(\lambda)$ and
$\phi(\lambda)$   as
\begin{align}
% \psi(\lambda) & =\ s_\lambda -  I_\lambda \lambda - t,  \label{eq:phi} \\
% \text{and }
  \phi(\lambda) & =\ (I_\lambda  -t^2)(I_\lambda \lambda- 2 s_\lambda)\lambda + s^2_\lambda-t^2 w_\lambda.\label{eq:psi}
\end{align}
For $I_j$, $s_j$, $w_j$, $j=1,\ldots, k$,    define
\begin{equation} \label{eq:quadF}
\varphi_j(\lambda):=(I_j-t^2)\left(I_j\lambda - 2s_j\right)\lambda + s_j^2 - t^2w_j, \  j=1, \ldots,k.
\end{equation}
 For brevity,   let
 $j_t=\min\{j: I_j\ge t^2, j=1,\ldots, k\}$
 which must exist by the fact that
$I_k =|\Ical_{\lambda_k}| = n$  and $n\ge t^2$.

The properties of $\varphi_j$ dependent on $j_t$ are analyzed    below.
 \begin{lemma}\label{lem.QuadFam}
\begin{enumerate}
\item[(i)]  For $j<j_t$, $\varphi_j$ is concave on $\R$ and strictly increasing on $(-\infty, \lambda_{j}]$.
\item[(ii)]
If $j\ge j_t$ and $I_{j_t} > t^2$, $\varphi_j$ is convex and strictly decreasing on $(-\infty, \lambda_j]$.
If $j > j_t$ and $I_{j_t} = t^2$, $\varphi_j$ is convex on $\R$ and strictly decreasing on $(-\infty, \lambda_j]$ and
\begin{equation}\label{eq:psicons}
\varphi_{j_t}(\lambda) \equiv s_{j_t}^2 - t^2w_{j_t}\le 0, \ \forall \lambda\in \R,
\end{equation}
where the equality holds  only if $I_1=t^2$.
\item[(iii)] For  $j> j_t, \varphi_j(\lambda_j)=\varphi_{j-1}(\lambda_j)$ and $\varphi_j(\lambda)>\varphi_{j-1}(\lambda)$, for any $\lambda<\lambda_j$.
\item[(iv)] For $j\ge j_t$, the smaller root for $\varphi_j(\lambda)=0$ is
\begin{equation}\label{eq:QuadRoot}
\lambda^\varphi_j = \frac{1}{I_j}\left( s_j  - t\sqrt{\frac{I_j w_j-s_j^2}{I_j-t^2}} \right).
\end{equation}
There is no root for $\varphi_j(\lambda)=0$ if $2\le j< j_t$.
\end{enumerate}
\end{lemma}

\begin{proof}
(i) It follows from \eqref{eq:quadF} that the first and second derivative of $\varphi_j$ is 
\begin{equation}\label{eq:Dpsi}
 \varphi_j'(\lambda) = 2(I_j-t^2)\left(I_j\lambda -  s_j \right) \  \text{ and}
%\end{equation}
% and the second derivative of $\varphi_j$ is
  % \[
 \quad \varphi_j''(\lambda)  = 2I_j(I_j-t^2).\end{equation}
Note that $ I_j\lambda - s_j < I_j \lambda_j- s_j \le 0$ for any $\lambda<\lambda_j$.
Therefore, both the sign of $\varphi_j'$ and $\varphi_j''$ are determined by the sign of $I_j-t^2$.
For  $j< j_t$, $\varphi''_j<0$ on $\R$ and $\varphi'_j(\lambda)>0$ on $(-\infty, \lambda_{j})$ since  $I_j<t^2$
by the definition of $j_t$.

(ii)
For $j\ge j_t$ and $I_j \ge I_{j_t} > t^2$, we have $\varphi''_j>0$ on $\R$ and $\varphi'_j(\lambda)<0$ on $(-\infty, \lambda_{j})$. % since $I_j>t^2$.
 For $j> j_t$ and $I_j > I_{j_t} = t^2$, we have  $\varphi''_j>0$ on $\R$ and $\varphi'_j(\lambda)<0$ on $(-\infty, \lambda_{j})$;
 in particular, $\varphi'_{j_t}(\lambda) =\varphi''_{j_t}(\lambda)  = 0$ and
  $\varphi_{j_t}$ takes  constant \eqref{eq:psicons} on $\R$ by the definition  \eqref{eq:quadF}.

(iii)
It   holds naturally  that
\[
 I_j =  \  I_{j-1}+(I_j-I_{j-1}), \
 s_j=  \ s_{j-1}+(I_j-I_{j-1})\lambda_j,\
  w_j=  \ w_{j-1}+(I_j-I_{j-1})\lambda_j^2.
  \]
%\begin{equation}\label{tmp.isphi}\begin{aligned}
% I_j = &\  I_{j-1}+(I_j-I_{j-1}), \\
% s_j= &\ s_{j-1}+(I_j-I_{j-1})\lambda_j,\\
%  \varphi_j= &\ \varphi_{j-1}+(I_j-I_{j-1})\lambda_j^2.
%  \end{aligned}
%  \end{equation}
Plugging this
% \eqref{tmp.isphi}
 into $\varphi_j(\lambda_j)$ yields that $\varphi_j(\lambda_j)=\varphi_{j-1}(\lambda_j)$.
 Moreover, it can be easily verified that
 % $\varphi'_j(\lambda)$, it can be check that
 for $j=2, \ldots, k$,
\[\varphi'_j(\lambda)-\varphi'_{j-1}(\lambda)=2(I_j-I_{j-1})\big[(I_j-t^2)(\lambda - \lambda_j)+ I_{j-1}\lambda-s_{j-1}\big].\]
If $j\ge j_t$, then $I_j-t^2\ge 0$, meaning $(I_j-t^2)(\lambda - \lambda_j)\le 0$ for $\lambda<\lambda_j$. In addition,
\[ I_{j-1}\lambda-s_{j-1} = \lambda  |\Ical_{j-1}| - \sum_{i\in\Ical_{j-1}} v_i < \lambda_j |\Ical_{j-1}| - \sum_{i\in\Ical_{j-1}} v_i  < 0\]
 for any $\lambda<\lambda_j$.
 Therefore, for $j\ge j_t$, it holds that
 $ \varphi'_j(\lambda)-\varphi'_{j-1}(\lambda)<0,\quad \text{for any } \lambda<\lambda_j.$
It then follows that $\varphi_j(\lambda) > \varphi_{j-1}(\lambda)$ for any $\lambda < \lambda_j$, completing
 the proof of (iii).

(iv)
The discriminant of $\varphi_j(\lambda) =0$ is
$\Delta = 4t^2(I_j-t^2)(I_j w_j - s_j^2).$
Now we discuss the sign of $\Delta$.
By the Cauchy-Schwarz inequality
\[ I_j w_j - s_j^2 = |\Ical_j| \sum_{i\in\Ical_j} v_i^2 -  \big(\sum_{i\in\Ical_j} v_i\big)^2 \ge 0,\]
where the inequality  holds strictly for  $j\ge 2$ since there are at least two distinct
values in the summation.  Therefore, if $j\ge j_t \ge 1$,
 then $\Delta\ge 0$ and the smaller root of $\varphi_j(\lambda)=0$ is given by \eqref{eq:QuadRoot}.
In particular, if $j=1$, then $\Delta =0$ and $\varphi_1(\lambda)$ has a unique root $\lambda_1$.
Moreover, if $j_t > j\ge 2$, then $\Delta<0$ since $I_j-t^2 < 0$ and $I_jw_j-s_j^2> 0$,
implying  $\varphi_j(\lambda)=0$ has no root.
This completes the proof of  (iv).
\end{proof}

\begin{proposition}\label{lem.psi}
The following statements hold true.
\begin{enumerate}
\item[(i)] $\phi$ is continuous on $\R$.
\item[(ii)] Suppose  $I_1>t^2$, $\phi$ is decreasing, piecewisely convex and quadratic on $(-\infty, v_{\max})$.
\item[(iii)] Suppose $I_1=t^2$, $\phi$ is decreasing, piecewisely convex and quadratic on $(-\infty, \lambda_2]$ and $\phi\equiv 0$ on $(\lambda_2, v_{\max})$.
\item[(iv)] Suppose $I_1< t^2$.  $\phi$ is increasing and piecewisely concave and quadratic on $[\lambda_{j_t}, v_{\max})$.
 Furthermore, if $I_{j_t}>t^2$, then $\phi(\lambda)$ is decreasing and piecewisely convex and quadratic  on $(-\infty, \lambda_{j_t})$; if $I_{j_t}=t^2$, then $\phi(\lambda)$ is decreasing and piecewise quadratic convex on $(-\infty, \lambda_{j_t+1})$, and   on $(\lambda_{j_t+1}, \lambda_{j_t}]$ % keeps negative constant % \eqref{eq:psicons}
 \[\varphi_{j_t}(\lambda) \equiv s_{j_t}^2 - t^2w_{j_t}.\]
\item[(v)] For any $\lambda\le \lambda_{j_t}$,
\begin{equation}\label{eq:MaxQuadFam}
\phi(\lambda)=\max\{q_{j_t}(\lambda), \ldots, q_k(\lambda)\},
\end{equation}
and $\phi(\lambda)$ is convex on $(-\infty, \lambda_{j_t}]$. Furthermore, for $j\ge {j_t}$ $\phi'(\lambda)=\varphi'_j(\lambda)$ for $\lambda\in (\lambda_{j+1},\lambda_j)$ and $\varphi'_j(\lambda_j)\in\partial \phi(\lambda_j)$.
\end{enumerate}
\end{proposition}
\begin{proof} Part (i) is trivial. % $\phi$ is continuous on $\R$ since it is the composition of the continuous functions. So Item (i) holds.

For part (ii), % we investigate the behaviors  of $\phi$ on intervals $(\lambda_{j+1}, \lambda_j], j=1,\ldots, k$ with $\lambda_{k+1}=-\infty $.
it can easily verified that
\begin{equation}\label{eq:psiPro3}
 \phi(\lambda) = \varphi_j(\lambda), \forall \lambda \in (\lambda_{j+1} ,  \lambda_{j}],\  j=1,..., k
\end{equation}
since \eqref{eq:index} and \eqref{eq:psi}. Moreover,  $\phi(\lambda)\equiv 0$ for $\lambda \in ( \lambda_1,+ \infty)$.

Part (iii) follows naturally  from Part (ii) and Lemma \ref{lem.QuadFam}(ii).

For part (iv), Lemma \ref{lem.QuadFam}(ii) shows % by  . Furthermore, from Lemma \ref{lem.QuadFam}(ii),
 $\varphi_j$ is convex for $j\ge j_t$.  Since Lemma \ref{lem.QuadFam}(iii),  $\phi$ takes form \eqref{eq:MaxQuadFam}.
Therefore,  $\phi$ is convex on $(-\infty, \lambda_{j_t}]$ with
$\varphi'_j(\lambda_j) \in \partial \phi(\lambda_j)$.
\end{proof}

Using Proposition~\ref{lem.psi}, we can summarize the behavior of $\phi$ as follows.
\begin{proposition}\label{pro.psiroot}
For  $\phi$, the following statements hold true:
\begin{enumerate}
\item[(i)]  If $I_1>t^2$, then $\phi(\lambda)>0$ for any $\lambda\in (-\infty, v_{\max})$.
\item[(ii)] If $I_1=t^2$, then $\phi(\lambda)>0$ for any $\lambda\in (-\infty, \lambda_2)$ and $\phi(\lambda)\equiv 0$ for any
$\lambda\in [\lambda_2, v_{\max})$.
\item[(iii)] If $I_1<t^2$,  $\phi(\lambda)=0$ possesses  a unique root  on $(-\infty,v_{\max})$ and this root  lies in $(-\infty,\lambda_{j_t})$. Furthermore, $\phi(\lambda)=0$ possesses  a unique root  on $(0,v_{\max})$ if and only if $\phi(0)>0$.
\end{enumerate}
\end{proposition}
\begin{proof}(i) If $I_1-t^2>0$, then $j_t=1$. By Proposition \ref{lem.psi}(ii), $\phi$ is strictly decreasing on $(-\infty, v_{\max})$. Therefore, part (i) is true.

(ii)
If $I_1-t^2=0$, then $j_t=1$ and $k>1$ since $I_1=t^2<n$. By Proposition \ref{lem.psi}(iii), $\phi$ is decreasing on $(-\infty, \lambda_2)$.  Hence part (ii) is true.   % \eqref{eq:psicons}(ii). So Item (ii) holds.

(iii)
If $I_1-t^2<0$, then $j_t\ge 2$ and $\phi$ is strictly increasing on $[\lambda_{j_t}, \lambda_1)$ by Proposition \ref{lem.psi} (iv).
Now we consider two cases.  If $I_{j_t}>t^2$, $\phi$ is decreasing on $(-\infty, \lambda_{j_t})$ by Proposition \ref{lem.psi} (iv); this
together with the fact $\phi$ is continuous and $\phi(\lambda_1)=0$, implies part (iii) is true. If $I_{j_t}=t^2$,
 $\phi$ is strictly decreasing on $(-\infty, \lambda_{j_t+1})$ and keeps a negative constant % in \eqref{eq:psicons}
 by Proposition \ref{lem.psi} (iv) because $j_t\ge 2$ and $n_{j_t}=t^2$. This implies that $\phi(\lambda)$ attains 0 only once  on $(-\infty, v_{\max})$,  and more precisely we know  the root lies in $(-\infty, \lambda_{j_t+1})$. Overall, we know part (iii) is true. % by Proposition \ref{lem.psi}(iii).
\end{proof}

\section{Characterizing the solution}\label{sec.solution}

In this section, we use  $\phi$   to characterize the
solution of \eqref{prob.b1b2}, \eqref{prob.b1s2} and \eqref{prob.s1s2}.
  Notice that  \eqref{prob.b1b2} % and \eqref{prob.l.b1b2} are  
is convex;  \eqref{prob.b1s2} and \eqref{prob.s1s2} are nonconvex.
We develop a unified framework using the partial Lagrangian duality,  which takes form 
\[ L(\bx, \lambda, \mu) = \tfrac{1}{2}\|\bx-\bv\|_2^2 + \lambda\Big(\sum_{i=1}^n x_i - t\Big) + \tfrac{\mu}{2}\Big(\sum_{i=1}^n x_i^2 - 1\Big). \]
Here for each problem the dual variables $\lambda$ is associated with the $\ell_1$ ball/sphere constraint 
and $\mu$ is associated with the $\ell_2$ ball/sphere  constraint, respectively.
The dual function is   given by
\begin{equation}\label{eq:dualF}\tag{$\mathcal{P}$}
g(\lambda, \mu) =\inf_{\bx\in\mathbb{R}_+^n} L(\bx, \lambda, \mu).
\end{equation}
The properties of $g$ are  analyzed in the following lemma.

\begin{lemma}\label{lemma.dualF}
For given $\lambda, \mu\in \R$, the following hold.
\begin{enumerate}
 \item[(i)] Suppose  $\mu=-1$.
  If $\lambda>v_{\max}$,  then  the optimal solution of \eqref{eq:dualF} is $\bx =0$; if $\lambda=v_{\max}$, then any $\bx$ satisfying
 \begin{equation}\label{eq:PriSol1}
 x_i\ge 0, i\in \mathcal{I}_1,\text{ and }\  x_i=0, i\notin \mathcal{I}_1;
\end{equation} 
is optimal. 
  In both cases, we have
 \[g(\lambda, \mu)=\frac{1}{2}\|\bv\|_2^2-\lambda t+\frac{1}{2}. \]

 \item[(ii)] Suppose $\mu>-1$.  The solution of \eqref{eq:dualF} is
\begin{equation}\label{eq:PriSol2}
 \bx(\lambda, \mu)=\frac{1}{1+\mu}(\bv-\lambda \mathbf{1})^+. 
\end{equation}
with dual function being
\begin{equation*}\label{eq:PriSol21}
g(\lambda, \mu)=\frac{1}{2}\|\bv\|_2^2-\lambda t-\frac{\mu}{2}-\frac{1}{2(1+\mu)}\|(\bv-\lambda \mathbf{1})^+\|^2_2 % \sum_{i=1}^n(\max(v_i-\lambda, 0))^2,
\end{equation*}
and  partial derivative
\begin{align}
%\frac{\partial g}{\partial \lambda} & =\frac{1}{1+\mu}\sum_{i=1}^n\max(v_i-\lambda, 0)-t =\frac{1}{1+\mu} \| (\bv -\lambda\mathbf{1})^+\|_1 - t,
\frac{\partial g}{\partial \lambda} & =\frac{1}{1+\mu} \| (\bv -\lambda\mathbf{1})^+\|_1 - t,
\label{eq:ParLamd} \\
% \frac{\partial g}{\partial \mu} & =\frac{1}{(1+\mu)^2}\sum_{i=1}^n(\max(v_i-\lambda, 0))^2-1 =\frac{1}{(1+\mu)^2} \| (\bv -\lambda\mathbf{1})^+\|_2^2 - 1.
 \frac{\partial g}{\partial \mu} &=\frac{1}{(1+\mu)^2} \| (\bv -\lambda\mathbf{1})^+\|_2^2 - 1. \label{eq:ParMu}
\end{align}
Moreover, $\nabla g(\lambda^*, \mu^*) = 0$ if and only if $\phi(\lambda^*)=0$ with $\lambda^* \in (-\infty, v_{\max})$ and $\mu^* = \|(\bv- \lambda^* \boldsymbol{1})^+\|_2 - 1$. In addition, $\bx(\lambda^*, \mu^*)$ reduces to $\bx^*= (\bv -\lambda^*\mathbf{1})^+/\|(\bv -\lambda^*\mathbf{1})\|_2$.
 \item[(iii)] If $\mu < -1$ or $\mu=-1$ and $\lambda < v_{\max}$, then $g(\lambda, \mu) = -\infty$.
\end{enumerate}
\end{lemma}

\begin{proof}(i) Suppose $\mu = -1$.  We have
$L(\bx,\lambda,\mu) = ( \lambda\mathbf{1} -\bv)^T\bx- \lambda t  - \frac{\mu}{2} + \frac{1}{2} \|\bv\|_2^2. $
Clearly, if $\lambda > v_{\max}$, the optimal solution  of \eqref{eq:dualF} is $\bx =0$; if
$\lambda =  v_{\max}$, the solution must satisfy \eqref{eq:PriSol1}. The rest of (i) is trivial.

(ii) Suppose $\mu > -1$.  Let $\boldsymbol{\zeta}\in\mathbb{R}^n$ be the multipliers for $\bx\in\mathbb{R}_+^n$.  The optimal $(\bx, \boldsymbol{\zeta})$ must satisfy
\[ \bx- \bv + \lambda \boldsymbol{1} + \mu \bx - \boldsymbol{\zeta} = 0,\ \bx^T\boldsymbol{\zeta} = 0, \  \bx \ge 0,\  \boldsymbol{\zeta}\ge 0.\]
If $x_i > 0$, meaning $\zeta_i = 0$, it follows that
$x_i = \frac{1}{1+\mu} (v_i - \lambda)$.  If $x_i = 0$, it follows that
$v_i - \lambda =  - \zeta_i \le 0$.
Therefore, \eqref{eq:PriSol2} is  true, and \eqref{eq:PriSol21}, \eqref{eq:ParLamd}  and \eqref{eq:ParMu}
can be computed accordingly.

Now, suppose $(\lambda^*, \mu^*)$ is stationary for $g$.  It holds that $\nabla g(\lambda^*, \mu^*) = 0$, implying $\lambda^* \in (-\infty, v_{\max})$ and
\[\phi(\lambda^*) = (1+\mu^*)^2 [\frac{\partial }{\partial \lambda}g(\lambda^*, \mu^*) + t]^2 - (1+\mu^*)^2 t^2 [\frac{\partial }{\partial \mu}g(\lambda^*, \mu^*)+1]^2 =  0.\]
Conversely, if $\phi(\lambda^*) = 0$ with $\lambda^* \in (-\infty, v_{\max})$,  letting
\[\mu^* = \sqrt{\| (\bv - \lambda^* \mathbf{1})^+\|_2} - 1,\]
we can see $\frac{\partial  }{\partial \lambda}g(\lambda^*, \mu^*)  = 0$ and $\frac{\partial  }{\partial \mu}g(\lambda^*, \mu^*) = 0$. Hence $(\lambda^*, \mu^*)$ is stationary
for $g$.
 This completes the proof of part (ii).

(iii) It can be verified trivially.

\end{proof}

\subsection{Projection onto $\mathbb{B}_1^t \cap \mathbb{B}_2 \cap \R^n_+$}

We first use the dual to analyze the properties of the solution of \eqref{prob.b1b2}.
Consider the Lagrangian dual problem of \eqref{prob.b1b2} 
\begin{equation}\label{eq:DualnonneBB}\tag{$\mathcal{D}_1$}
\begin{array}{ll}
{\rm maximize} & g(\lambda, \mu)\\
{\rm subject\ to}  & \lambda\ge 0, \mu \ge 0.
\end{array}
\end{equation}
Let $(\lambda^*,\mu^*)$ solve dual   \eqref{eq:DualnonneBB}. If the solution $\bx^*$ of \eqref{eq:dualF} for given $(\lambda^*,\mu^*)$ is feasible for \eqref{prob.b1b2} and satisfies the complementary condition
\begin{equation}\label{eq:nonneBBCC}
\lambda^*(\mathbf{1}^T\bx^*-t)=0 \text{and}\ \ \mu^*(\|\bx^*\|_2^2-1)=0,
\end{equation}
then $g(\lambda^*,\mu^*)=\tfrac{1}{2}\|\bx^*-\bv\|_2^2.$ We know $\bx^*$ solves \eqref{prob.b1b2}.
By the first-order optimality condition, $(\lambda^*,\mu^*)$ solves \eqref{eq:DualnonneBB} if and only if
\begin{equation}\label{eq:OptCforD}
\frac{\partial g(\lambda^*, \mu^*)}{\partial \lambda}(\lambda-\lambda^*)+ \frac{\partial g(\lambda^*, \mu^*)}{\partial \mu} (\mu-\mu^*)\le 0, \ \text{ for any }  \lambda\ge 0, \mu\ge 0.
\end{equation}

\begin{theorem}\label{thm.nonnega}
Let $\bx$ be the optimal solution of \eqref{prob.b1b2}. Then one of the following statements must be true:
\begin{enumerate}
\item[(i)]  $\|\bv^+\|_2 \le 1$ and $\|\bv^+\|_1 \le t$. In this case,  $\bx= \bv^+$.
\item[(ii)]  $\| \bv^+\|_2 > 1$ and $\| \bv^+\|_1 \le t\|\bv^+\|_2$.  In this case, $\bx = \bv^+/\|\bv^+\|_2$.
\item[(iii)] $\|\bv^+\|_1>t$ and $\|\bv^+\|_1 > t\|\bv^+\|_2$. In this case,  $\psi(\lambda)=0$ has a unique root $\hat\lambda$ in $(0,v_{\max})$.  Furthermore, if
$\|(\bv-\hat\lambda\mathbf{1})^+\|_2 \le 1$,  then $\bx = (\bv-\hat\lambda\mathbf{1})^+$; Otherwise,
$\phi(\lambda)=0$ has a unique root  $\lambda^*$ in $(0,\hat\lambda)$, and
\begin{equation}\label{eq:SolPri}
\bx = (\bv -\lambda^*\mathbf{1})^+/\|(\bv -\lambda^*\mathbf{1})\|_2.
\end{equation}
\end{enumerate}
\end{theorem}

\begin{proof}
Case (i).
If $\|\bv^+\|_2 \le 1$ and $\|\bv^+\|_1 \le t$, we can see  $\lambda^*=0, \mu^*=0$ satisfies the optimality condition \eqref{eq:OptCforD} by \eqref{eq:ParLamd} and \eqref{eq:ParMu}. In this case, by \eqref{eq:PriSol2}, the solution of \eqref{eq:dualF} is $\bx^*=\bv^+$. Obviously  $\bv^+$ is feasible for \eqref{prob.b1b2}, and  $(\lambda^*, \mu^*, \bx^*)$ % = (0, 0, \bv^+)$
satisfies \eqref{eq:nonneBBCC}.

Case (ii).
If $\| \bv^+\|_2 > 1$ and $\| \bv^+\|_1 \le t\|\bv^+\|_2$, we can see  that $\lambda^*=0, \mu^*=\|\bv^+\|_2-1$ satisfies the optimality condition \eqref{eq:OptCforD} by \eqref{eq:ParLamd} and \eqref{eq:ParMu}.
In this case, by \eqref{eq:PriSol2}, the solution of \eqref{eq:dualF} is $\bx^*=\bv^+/\|\bv^+\|_2$. Obviously $\bx^*$ is feasible for \eqref{prob.b1b2}, and
 $(\lambda^*, \mu^*, \bx^*)$ % = (0, \|\bv^+\|_2-1, \bv^+/\|\bv^+\|_2)$
  satisfies \eqref{eq:nonneBBCC}.

Case (iii).
Now consider the case that neither (i) nor (ii) happens.
 Since case (i) is not satisfied,  we know $\bv$ must satisfy
 \[ \|\bv^+\|_2 > 1 \quad\text{ or} \quad \|\bv^+\|_1 > t.\]
On the other hand, case (ii) is  not satisfied,  so that we know $\bv$ must satisfy
\[ \|\bv^+\|_2 \le 1\quad \text{or} \quad \|\bv^+\|_1 > t \|\bv^+\|_2.\] There are four subcases  to consider.  (a) $\|\bv^+\|_2>1$ and $\|\bv^+\|_2\le 1$, which can
never happen.  (b) $\|\bv^+\|_2>1$ and $\|\bv^+\|_1 > t\|\bv^+\|_2$, implying $\|\bv^+\|_1 > t$.  (c)  $\|\bv^+\|_1 > t$ and $\|\bv^+\|_2 \le 1$, indicating
$\|\bv^+\|_1 > t \ge t\|\bv^+\|_2$.  (d) $\|\bv^+\|_1>t$ and $\|\bv^+\|_1 > t\|\bv^+\|_2$.

Overall, we have shown that if $\bv$ does not satisfy Case  (i) nor (ii), then it must be true that   $\|\bv^+\|_1>t$ and $\|\bv^+\|_1 > t\|\bv^+\|_2$.
It follows that  $\psi(0)=\|\bv^+\|_1-t>0$ and $\phi(0)=\|\bv^+\|_1^2-t^2\|\bv^+\|_2^2>0$. From Proposition \ref{lem.phi}, $\psi=0$ has a unique root $\hat\lambda$ on  $(0,\lambda_1)$.  We  consider two situations.

First, if $\|(\bv-\hat\lambda\mathbf{1})^+\|_2 \le 1$, we can see  that $(\lambda, \mu) = (\hat\lambda, 0)$ satisfies the optimality condition \eqref{eq:OptCforD} by \eqref{eq:ParLamd} and \eqref{eq:ParMu}. Therefore, by \eqref{eq:PriSol2}, the solution of \eqref{eq:dualF} is $\bx=(\bv-\hat\lambda\mathbf{1})^+$. Obviously,  $(\bx^*, \lambda^*, \mu^*) =((\bv-\hat\lambda\mathbf{1})^+, \hat\lambda, 0)$ satisfy the complementary slackness and  is primal-dual feasible for \eqref{prob.b1b2}.

Second,  if $\|(\bv-\hat\lambda\mathbf{1})^+\|_2 > 1$, then
\[ \phi(\hat\lambda) = \| (\bv - \hat\lambda\mathbf{1})^+\|_1^2 - t^2\|(\bv-\hat\lambda\mathbf{1})^+\|_2^2 = t^2(1-\|(\bv-\hat\lambda\mathbf{1})^+\|_2^2) < 0,\]
where the second equality follows $\psi(\hat\lambda)=0$, i.e. $\|(\bv-\hat\lambda\bm 1)^+\|_1=t$.
It follows from Proposition \ref{pro.psiroot}  that $\phi(\lambda)=0$ has a unique root $\lambda^*$ on $(0, \hat\lambda)$, meaning  $\lambda^*<\hat\lambda$. This implies $\| (\bv-\lambda^*\mathbf{1})^+\|_2 \ge \| (\bv-\hat\lambda\mathbf{1})^+\|_2 > 1$.
We can see that $(\lambda^*, \mu^*)$, with $\mu^*=  \|(\bv-\lambda^*\mathbf{1})^+\|_2-1$ satisfies the optimality condition \eqref{eq:OptCforD} by \eqref{eq:ParLamd} and \eqref{eq:ParMu}, and it follows from
 \eqref{eq:PriSol2} that the solution of \eqref{eq:dualF} is \eqref{eq:SolPri}.
Therefore,  $\|\bx^*\|_2 = 1$ and
\[ \mathbf{1}^T\bx^* = \| (\bv - \lambda^*\mathbf{1})^+ \|_1/\|(\bv-\lambda^*\mathbf{1})^+\|_2 = t\]
from  $\phi(\lambda^*)=0$. Obviously,  $(\bx^*, \lambda^*, \mu^*) = ((\bv -\lambda^*\mathbf{1})^+/\|(\bv -\lambda^*\mathbf{1})\|_2,  \lambda^*, \|(\bv-\lambda^*\mathbf{1})^+\|_2-1)$  satisfies \eqref{eq:nonneBBCC}  and is primal-dual feasible for \eqref{prob.b1b2}.  This completes the proof.
\end{proof}

Table~\ref{tab.1} summarizes the four cases (Case (iii) is further split into two subcases) described in Theorem~\ref{thm.nonnega},
which are represented by the four regions illustrated  in $\mathbb{R}^2$ in  Figure~\ref{fig.1}.

Region $C_{\text{o}}$:  If $\bv^+\|_1 \le t$ and $\|\bv^+\|_2\le 1$,  this indicates that $\bv$ is in the feasible region. The projection of $\bv$
is simply $\bv$.  This corresponds to Case (i).

Region $C_{\text{I}}$:
If $\|\bv\|_1 \le t^2\|\bv\|_2$, in $\mathbb{R}^2$,
this condition  is equivalent to
\[ (t^2-1)v_2 \le  (1-t\sqrt{2-t^2})v_1\ \text{ or }\  (t^2-1) v_2 \ge ( {1+t\sqrt{2-t^2}}) v_1\] (note that it is assumed
$t < \sqrt{2}$).
If $\bv$ lies in region $C_{\text{I}}$, then the projection of $\bv$ reverts to the projection onto
the $\ell_2$ ball, which is simply $\bv/\|\bv\|_2$.  This corresponds to Case (ii).

Now suppose  $\bv$ does not lie in region $C_{\text{o}}$ nor $C_{\text{I}}$. This
corresponds to the situation of  Theorem \ref{thm.nonnega}(iii), i.e.,
$\|\bv\|_1 > t$ and $\|\bv\|_1 > t \|\bv\|_2$.
It follows that
$\psi(\lambda)  = 0$ has a positive root $\hat\lambda$.
%Now consider the solution  for
%\begin{align}
%\psi(\lambda) & =\max(v_1-\lambda, 0)+\max(v_2-\lambda,0) - t = 0,\label{eq.temp.1}\\
%\phi(\lambda) & =[\max(v_1-\lambda, 0)]^2+[\max(v_2-\lambda,0)]^2 - 1 = 0.\label{eq.temp.2}
%\end{align}
%
%If $v_1 \le \hat \lambda \le v_2$, then \eqref{eq.temp.1} implies that $\psi(\hat\lambda) = v_2 - \hat \lambda = t$ and
%$\|(\bv-\hat\lambda\mathbf{1})^+\|_2  =  (v_2 - \hat \lambda) = t^2 >1$. Similarly, if $v_2 \le \hat \lambda \le v_1$,
%then $\|(\bv-\hat\lambda\mathbf{1})^+\|_2 > 1$.
%
%

Region $C_{\text{II}}$:
if $v_1-\sqrt{2-t^2} \le v_2 \le v_1 + \sqrt{2-t^2}$, then
$v_1 \ge \frac{v_1+v_2-t}{2}$ and $v_2 \ge \frac{v_1+v_2-t}{2}$.
It is easy to see that
$\hat \lambda $ satisfies $\psi(\hat\lambda)  = t$ and
$\|(\bv-\hat\lambda\mathbf{1})^+\|_2 \le 1$.
This corresponds to the region $C_{\text{II}}$ in Figure~\ref{fig.1}, and the
projection of $\bv$ is simply the projection onto the $\ell_1$ ball, which is
$(\bv-\hat\lambda\mathbf{1})^+$.
This corresponds to Case (iii)-a.

%
%
%, meaning
%$ \frac{1-t\sqrt{2-t^2}}{t^2-1}v_1 < v_2 <\frac{1+t\sqrt{2-t^2}}{t^2-1}v_1$
%
%  then
%$\psi(\lambda)=\max(v_1-\lambda, 0)+\max(v_2-\lambda,0) - t = 0$ has a positive root $\tfrac{v_1+v_2-t}{2}$.  Rearranging
%\[ (v_1-\tfrac{v_1+v_2-t}{2})^2 + (v_2 - \tfrac{v_1+v_2-t}{2})^2 \le 1,\]
%we obtain $(v_2-v_1)^2 \le 2 -t^2$, or equivalently, $v_1-\sqrt{2-t^2} \le v_2 \le v_1 + \sqrt{2-t^2}$, depicted by the region $C_{\text{II}}$ in Figure~\ref{fig.1}.
%Therefore, if $\bv$ is in region $C_{\text{II}}$ and $\|\bv\|_1 > t$, the projection of $\bv$ is on the $\ell_1$ sphere.

Region $C_{\text{III}}$:
if $(v_2-v_1)^2 > 2-t^2$ and $\|\bv\|_1^2 > t^2\|\bv\|_2^2$,  the projection of $\bv$ onto $\ell_1$ ball is outside the $\ell_2$ ball, i.e.,
$\|(\bv-\hat\lambda\mathbf{1})^+\|_2 > 1$. This corresponds to Case (iii)-b.

%
%
%
%$\lambda^* = $ is the root of
%$ \phi(\lambda) =  0$.
%
%
%If $\bv$ lies in region  $C_{\text{III}}$,
%the projection falls onto the intersection of the $\ell_1$ sphere and the $\ell_2$ sphere.
%
%

\begin{table}[htb]
\centering
\caption{Four cases of $( \lambda^*, \mu^*, \bx^*)$ where $\hat \lambda$ is the root of $\psi$.}\label{tab.1}
\begin{tabular}{lcclcl}\hline
& Case & $\bx^*$ & $\lambda^*$  &  $\mu^*$  & Region \\ \hline
(i) &  $\|\bv^+\|_1\le t, \|\bv^+\|_2\le 1$   &   $\bv^+$   &  0   &  0 & $C_{\text{o}}$ \\
(ii) & $\|\bv^+\|_2>1, \|\bv^+\|_1\le t\|\bv^+\|_2$  &  $\bv^+/\|\bv^+\|_2$   &  0   &  $\|\bv^+\|_2-1$ & $C_{\text{I}}$ \\
(iii)-a & $\|(\bv-\hat\lambda\mathbf{1})^+\|_2\le1$ & $(\bv-\hat\lambda\mathbf{1})^+$   &   $\hat\lambda$   &   0 & $C_{\text{II}}$\\
(iii)-b & $\|(\bv-\hat\lambda\mathbf{1})^+\|_2 > 1$ & $\frac{(\bv-\lambda^*\mathbf{1})^+}{\|(\bv-\lambda^*\mathbf{1})^+\|_2}$  &   $\lambda^*$  &  $\|(\bv-\lambda^*\mathbf{1})^+\|_2-1$ & $C_{\text{III}}$ \\ \hline
\end{tabular}
\end{table}

\begin{figure}[htb]
\centering
\includegraphics[scale=0.36]{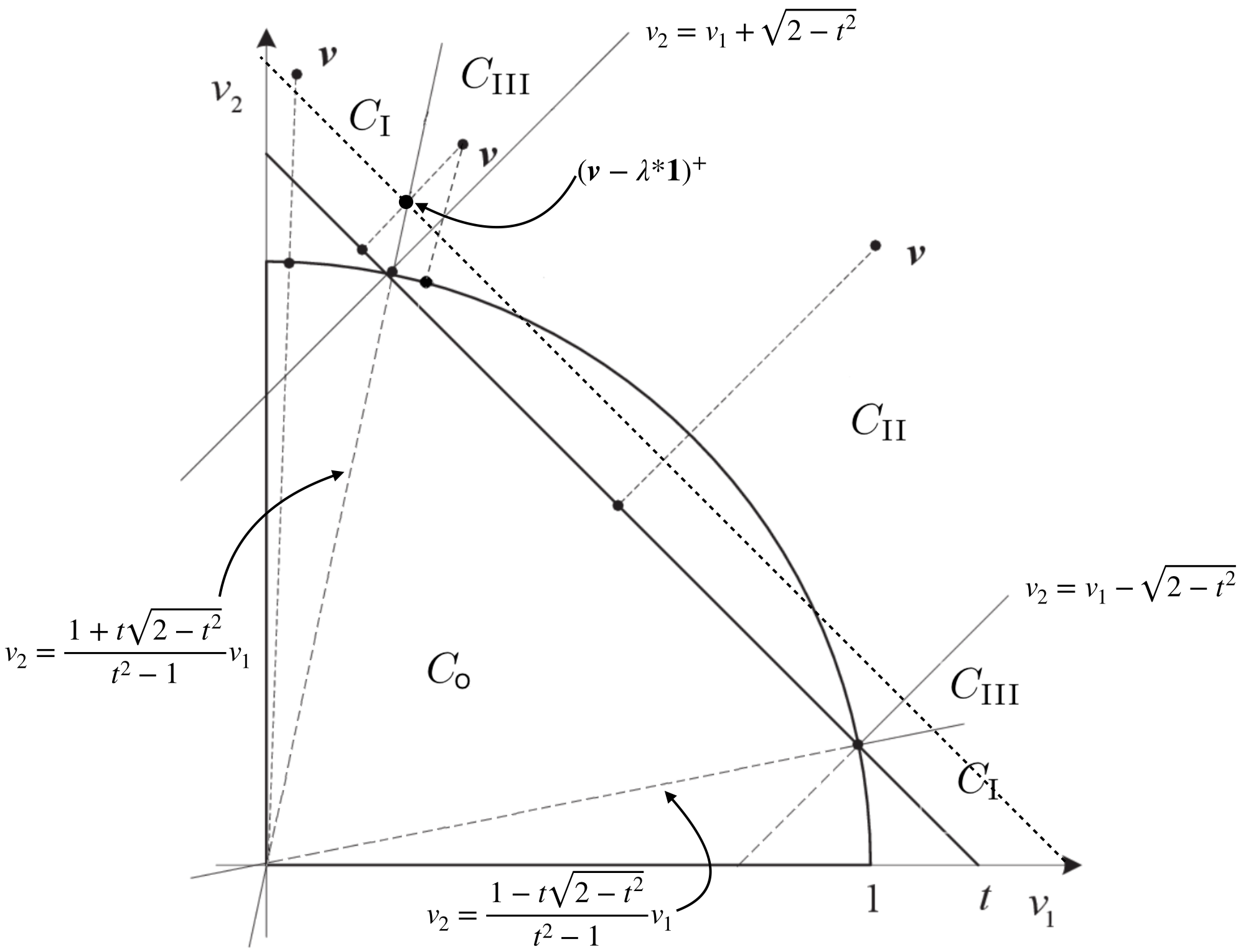}
\caption{The projection of $\bv$ in $\mathbb{R}^2$.}
\label{fig.1}
\end{figure}

\subsection{Projection onto $\mathbb{S}_1^t \cap \mathbb{S}_2 \cap \R^n_+$}

We characterize the projections onto $\mathbb{S}_1^t \cap \mathbb{S}_2 \cap \R^n_+$, i.e., the optimal solution of \eqref{prob.s1s2}.
By Theorem \ref{lemma.dualF}, we can consider the Lagrangian dual problem of \eqref{prob.s1s2},
\begin{equation}\label{eq:DualnonneSS}\tag{$\mathcal{D}_2$}
\begin{array}{ll}
{\rm maximize} & g(\lambda, \mu)\\
{\rm subject\ to}  & \mu \ge -1.
\end{array}
\end{equation}
and let $(\lambda^*,\mu^*)$ solve the dual   \eqref{eq:DualnonneSS}. If the solution $\bx^*$ of \eqref{eq:dualF} for given
 dual feasible  $(\lambda^*,\mu^*)$ is primal feasible for \eqref{prob.s1s2}, then $g(\lambda^*,\mu^*)=\tfrac{1}{2}\|\bx^*-\bv\|_2^2$. We get
 $\bx^*$ solves \eqref{prob.s1s2}.  Therefore, for such $\bx^*$, we only have to verify its satisfaction of the constraints of \eqref{prob.s1s2}.

\begin{theorem}\label{thm.nonneSS}
For any $\bv\in \R^n$, one of the following statements must be true:
\begin{enumerate}
\item[(i)] If $I_1>t^2$, the solution of \eqref{prob.s1s2} is not unique. Any solution of the system
\begin{equation}\label{proj.MSnonneSS}
\sum_{i\in \mathcal{I}_1}x_i=t, \sum_{i\in \mathcal{I}_1}x_i^2=1,x_i\ge 0, i\in \mathcal{I}_1, x_i=0, i\notin \mathcal{I}_1
\end{equation}
solves \eqref{prob.s1s2}.

\item[(ii)]  If $I_1=t^2$, then \eqref{prob.s1s2} has the unique solution
\[ x_i=1/\sqrt{I_1}, i\in \mathcal{I}_1\ \text{ and }\  x_i=0,  i\notin \mathcal{I}_1,\] which is also the unique solution of \eqref{proj.MSnonneSS}.

\item[(iii)] If $I_1<t^2$, then \eqref{prob.s1s2} has the unique solution $\bx = (\bv -\lambda^*\mathbf{1})^+/\|(\bv -\lambda^*\mathbf{1})^+\|_2,$
where $\lambda^*$ is the unique root of $\phi(\lambda)=0$ on $(-\infty, v_{\max})$.
\end{enumerate}
\end{theorem}

%By Lemma \ref{lemma.dualF}(ii),  $g(\lambda, \mu)$ has a stationary point in the region $(0, \infty)\times   (-1, \infty)$ if and only if $\phi(\lambda)=0$ has a root in $(0, v_{\max})$.
%
%

\begin{proof}  (i) If $I_1>t^2$, by Proposition \ref{pro.psiroot}(i), $\phi(\lambda)$ has no root on $(0, v_{\max})$.
By Lemma \ref{lemma.dualF}(ii),
$g(\lambda, \mu)$ has no stationary point in the region $\R\times (-1, \infty)$.
Therefore,  the optimal solution of \eqref{eq:DualnonneSS} is $\lambda^*=\lambda_1, \mu^*=-1$.
By Lemma \ref{lemma.dualF}(i),  any point $\bx$ satisfying \eqref{proj.MSnonneSS} must
be a solution of  \eqref{eq:dualF}, since it satisfies \eqref{eq:PriSol1}. Moreover, it satisfies the constraints of \eqref{prob.s1s2},
so that it is optimal. This completes the proof of   (i).
%
%
%
%By Lemma \ref{lemma.dualF}(i),   $\lambda^*=v_{\max}, \mu^*=-1$
%consider those solution of \eqref{eq:dualF} which satisfies \eqref{proj.MSnonneSS}. It is obviously that those $\bx$ are feasible for \eqref{prob.s1s2}. So   (i) holds.

  (ii)  If  $I_1=t^2$, $\phi(\lambda)\equiv 0$ on $[\lambda_2,v_{\max})$ by Proposition \ref{pro.psiroot}(ii), we can see  that any $(\lambda,\mu)$ with $\lambda\in [\lambda_2, \lambda_1)$ and $\mu=\|(\bv-\lambda\bfo)^+\|_2 - 1$ is stationary for the dual   \eqref{eq:DualnonneSS}. For those $(\lambda, \mu)$, by Lemma \ref{lemma.dualF}(ii), the solution of  \eqref{eq:dualF}
satisfies $x_i(\lambda, \mu) > 0$, $i\in\Ical_1$ and $x_i(\lambda, \mu) = 0$, $i\notin\Ical_1$.
If we further require the satisfaction of
the constraints of  \eqref{prob.s1s2}, then there is only
a unique point
$x_i^*=\tfrac{1}{\sqrt{I_1}}$ for $i\in \mathcal{I}_1$ and $x_i^*=0$ for $i\notin \mathcal{I}_1$ since $I_1=t^2$. This completes the proof of   (ii).

  (iii) If $I_1<t^2$, $\phi(\lambda)=0$ has a unique root $\lambda^*$ in $(-\infty, v_{\max})$ by Proposition \ref{pro.psiroot}(iii).
Therefore, by \eqref{eq:ParMu}
  the optimal solution $(\lambda^*, \mu^*)$ of \eqref{eq:DualnonneSS} satisfies  $\mu^*=\|(\bv-\lambda^*\bfo)^+\|_2 - 1$.
  Given $(\lambda^*, \mu^*)$, by \eqref{eq:PriSol2}, the solution of \eqref{eq:dualF} is given in \eqref{eq:SolPri}.
It can be easily verified   that $\bx^*$  is feasible for \eqref{prob.s1s2}.  Therefore,
  $\bx^*$ solves \eqref{prob.s1s2}, completing the proof of  (iii).
\end{proof}

\subsection{Projection onto $\mathbb{B}_1^t \cap \mathbb{S}_2 \cap \R^n_+$}

We consider the Lagrangian dual problem of \eqref{prob.b1s2}
\begin{equation}\label{eq:DualnonneBS}\tag{$\mathcal{D}_3$}
\begin{array}{ll}
{\rm maximize} & g(\lambda, \mu)\\
{\rm subject\ to}  & \lambda\ge 0, \mu \ge -1.
\end{array}
\end{equation}
and let $(\lambda^*,\mu^*)$ be optimal for the dual   \eqref{eq:DualnonneBS}.
 If the optimal solution $\bx^*$   of \eqref{eq:dualF} for given dual feasible
  $(\lambda^*,\mu^*)$ is also primal  feasible for \eqref{prob.b1s2}
  and    $(\lambda^*, \mu^*)$ and $\bx^*$ satisfy complementarity
  \begin{equation}\label{eq:nonneBSCC}
  \lambda^*(\mathbf{1}^T\bx^*-t)=0,
  \end{equation}
   then $g(\lambda^*,\mu^*)=\tfrac{1}{2}\|\bx^*-\bv\|_2^2$. So $\bx^*$ is optimal for \eqref{prob.b1s2}.

We characterize the projections onto $\mathbb{B}_1^t \cap \mathbb{S}_2 \cap \R^n_+$ below. % , i.e., the optimal solution of \eqref{prob.b1s2}.

\begin{theorem}\label{thm.nonneBS}
For any $\bv\in \R^n$, one of the following statements must be true:
\begin{enumerate}
     \item[(i)] Suppose $v_{\max}>0$,  $I_1\le t^2$ and $\|\bv^+\|_1>t\|\bv^+\|_2$.  Then \eqref{prob.b1s2}  has a unique solution
$(\bv -\lambda^*\mathbf{1})^+/\|(\bv -\lambda^*\mathbf{1})\|_2,$ %     \eqref{eq:SolPri}
     where $\lambda^*$ is the unique root of $\phi(\lambda)=0$ on $(0, v_{\max})$.

      \item[(ii)] Suppose $v_{\max}>0$,  $I_1\le t^2$ and $\|\bv^+\|_1 \le t\|\bv^+\|_2$.  Then \eqref{prob.b1s2}  has a unique solution
      $\bv^+/\|\bv^+\|_2$.

     \item[(iii)] Suppose  $v_{\max}>0$ and $I_1>t^2$.  Then the solution of \eqref{prob.b1s2}  is not unique.
     Any point satisfying  \eqref{proj.MSnonneSS} is the optimal solution of \eqref{prob.b1s2}.
     \item[(iv)] Suppose $v_{\max}=0$. The solution of \eqref{prob.b1s2}  is not unique. Any $\bx$ satisfying
         \begin{equation}\label{proj.MSnonneBS1}
            \sum_{i\in \mathcal{I}_1}x_i \le t, \sum_{i\in \mathcal{I}_1} x_i^2 = 1, x_i\ge 0, i\in \mathcal{I}_1, x_i=0, i\notin \mathcal{I}_1
         \end{equation}
         is optimal for \eqref{prob.b1s2}.
     \item[(v)] Suppose $v_{\max}<0$. Any $\bx$ satisfying $x_i=1$ for some $i\in \mathcal{I}_1$ and $x_j=0$ for $j\neq i$ is
     optimal for  \eqref{prob.b1s2}.
\end{enumerate}
\end{theorem}

\begin{proof}  (i)  If $I_1 = t^2$, by Proposition \ref{pro.psiroot}(ii), $\phi(\lambda)=0$ for any $\lambda\in [\lambda_2, v_{\max})$.
If $I_1 > t^2$, notice that $\|\bv^+\|_1>t\|\bv^+\|_2$ is equivalent to  $\phi(0) >0$;  by Proposition \ref{pro.psiroot}(iii),
$\phi(\lambda) = 0$ has a unique solution $\lambda^*$ on $(0, v_{\max})$. In both cases, we let $\mu^* =\|(\bv-\lambda^*{\bf 1})^+\|_2 -1$. Then by Lemma \ref{lemma.dualF}(ii),  we have  $\nabla g(\lambda^*, \mu^*)=0$. So $(\lambda^*, \mu^*)$ is the optimal solution of \eqref{eq:DualnonneBS}. In this case, $\bx^*=(\bv -\lambda^*\mathbf{1})^+/\|(\bv -\lambda^*\mathbf{1})\|_2$ is the solution of \eqref{prob.b1s2} and is feasible for \eqref{prob.b1s2}. In addition, $(\lambda^*,\mu^*, \bx^*)$ satisfies \eqref{eq:nonneBSCC}. This proves   (i).

  (ii)   Same argument as for   (i) implies that    $\mu^* > -1$.
Since $\|\bv^+\|_1 \le t\|\bv^+\|_2$,  letting $(\lambda^*, \mu^*)$ with $\lambda^*= 0$ and $\mu^*=\|\bv^+\|_2-1$,
we have
\begin{equation}\label{eq:OptCforD2}
\begin{aligned}
&\ \frac{\partial g(\lambda^*, \mu^*)}{\partial \lambda}(\lambda-\lambda^*)+ \frac{\partial g(\lambda^*, \mu^*)}{\partial \mu} (\mu-\mu^*)\\
=  &\ ( \|\bv^+\|_1/ \|\bv^+\|_2 -t)  \lambda+ 0 (\mu-\mu^*)\\
\le &\  0
\end{aligned}
\end{equation}
for any $\lambda\ge 0, \mu\ge 0$. Therefore,  $(\lambda^*, \mu^*)$ is the optimal solution of \eqref{eq:DualnonneBS}.
For given $(\lambda^*, \mu^*)$, by \eqref{eq:PriSol2} the solution of \eqref{eq:dualF} is $\bx^*=\bv^+/\|\bv^+\|_2$. We can see  that $\bx^*$  is feasible
for  \eqref{prob.b1s2}. $(\lambda^*,\mu^*, \bx^*)$ obviously satisfies \eqref{eq:nonneBSCC}.
This proves   (ii).

  (iii)
Let $v_{\max}>0$ and $I_1>t^2$. By Proposition \ref{pro.psiroot}(i), $\phi(\lambda)=0$ does not have
root on $(-\infty, v_{\max})$.  By Lemma \ref{lemma.dualF}(ii),
$g(\lambda, \mu)$ has no stationary point in the region $\R\times (-1, \infty)$.
Therefore,  the optimal solution of \eqref{eq:DualnonneSS} is $\lambda^*=\lambda_1, \mu^*=-1$.
By Lemma \ref{lemma.dualF}(i),  any point $\bx$ satisfying \eqref{proj.MSnonneSS} must
be a solution of  \eqref{eq:dualF}, since it satisfies \eqref{eq:PriSol1}. Moreover, it satisfies the constraints of \eqref{prob.s1s2}. In addition, any $\bx$ satisfying \eqref{proj.MSnonneSS} with $\lambda^*=0$ and $\mu^*=-1$ also satisfies \eqref{eq:nonneBSCC}. This completes the proof of   (iii).

%The optimal solution of the dual   \eqref{eq:DualnonneBS} is $(\lambda_1, -1)$. For given $(\lambda_1, -1)$, by \eqref{eq:PriSol1} the solution of \eqref{proj.MSnonneSS} solves \eqref{eq:dualF} and those $\bx$ are feasible for \eqref{prob.b1s2}  while $(\lambda_1, -1)$ and those $\bx$ satisfy the complementary slackness. So the solution of \eqref{proj.MSnonneSS} solves \eqref{prob.b1s2}. Item (ii) holds.

  (iv)
Let $\bv^+=0$ and $\lambda_1=0$.
It must be true that $\mu^* = -1$ since otherwise $\nabla g(\lambda, \mu) \neq 0$ for any $\lambda \ge 0$ and $\mu > -1$ by
  \eqref{eq:ParLamd} and \eqref{eq:ParMu}. This indicates by Lemma \ref{lemma.dualF}(i) that $\lambda^*=v_{\max}=0$.
  For   $(\lambda^*, \mu^*) = (0, -1)$, by \eqref{eq:PriSol1} the solution of \eqref{proj.MSnonneBS1} solves \eqref{eq:dualF} and those $\bx$ are feasible for \eqref{prob.b1s2}. Therefore, any $\bx$ satisfying  \eqref{proj.MSnonneBS1}  must be optimal for \eqref{prob.b1s2}.  This proves  (iv).

  (v)
Let $\bv^+=0$ and $\lambda_1<0$.
Same argument applied to the proof of   (iv) implies that
the optimal solution of the dual   \eqref{eq:DualnonneBS} is $(\lambda^*,\mu^*)=(0,-1)$.
However,    by \eqref{eq:PriSol1} the solution of \eqref{eq:dualF} is $\bx=0$ which is infeasible for \eqref{prob.b1s2}.
Therefore, we   further investigate the  primal optimal solution %So there is a dual gap between the primal   \eqref{prob.b1s2}  and it's dual   \eqref{eq:DualnonneBS}.
and rewrite \eqref{prob.b1s2}  as
\begin{equation}\label{eq:LinMaxNonneBS}
\mathop{\rm minimize}   -\bv^T\bx  \quad
\st   \ \ \bx\in\mathbb{R}^n_+\cap \mathbb{B}_1^t \cap \mathbb{S}_2.
\end{equation}
% where $v_i<0, i=1, \ldots, n$.
Since $-\bv^T\bx\ge (-\lambda_1) \boldsymbol{1}^T\bx$ for all $\bx\in\mathbb{R}_+^n$ and equality holds if and only if $x_i=0$ for $i\notin \mathcal{I}_1$ since $-\lambda_1>0$.
Hence, we can instead consider minimize $\boldsymbol{1}^T\bx$  and obtain
\begin{equation}
\label{prob.tmp.2ball}
\mathop{\rm minimize} \boldsymbol{1}^T\bx \quad \ \ \st\quad   \bx\in\mathbb{R}^n_+  \cap \mathbb{S}_2.
 \end{equation}
Notice the constraint $\sum_{i=1}^n x_i \le t$ is eliminated since  $\sum_{i=1}^n x_i$ is minimized in the objective;
this constraint must be satisfied at the optimal solution of  \eqref{prob.tmp.2ball} since otherwise \eqref{eq:LinMaxNonneBS}
is infeasible. The optimal solution of \eqref{prob.tmp.2ball} is obviously any $\bx$
with exactly one nonzero component, i.e., $x_i = 1$ for some $i\in\{1,\ldots,n\}$, and such
points are naturally  feasible for $\sum_{i=1}^n x_i \le t$ due to the assumption $t > 1$.
On the other hand, if we further require $i\in\Ical_1$, then
$-\bv^T\bx =   -\lambda_1 \boldsymbol{1}^T\bx$, meaning
$-\bv^T\bx$ attains the minimum of $ -\lambda_1 \boldsymbol{1}^T\bx$ on the same feasible region.
Therefore, any $\bx$ with only one nonzero component  $x_i = 1$ and $i\in \Ical_1$ is optimal for \eqref{prob.b1s2}.
Case (iv) is also true.
\end{proof}

\section{Proposed  Algorithms for Solving   $\phi=0$}\label{algorithm}\label{sec.algorithm}

 Based on Theorems \ref{thm.nonnega}, \ref{thm.nonneSS} and \ref{thm.nonneBS}, % and   \ref{th:LinMaxNonBB},
 the projections onto  $\mathbb{B}_1^t \cap \mathbb{B}_2$,    $\mathbb{S}_1^t\cap \mathbb{S}_2$ and   $\mathbb{B}_1^t \cap \mathbb{S}_2$,
as well as the problem of maximizing a linear function over the $\mathbb{B}_1^t \cap \mathbb{B}_2$ can all revert to
solving the equation of $\phi(\lambda) = 0$. If we can design an algorithm to quickly  find the root, then
all the problems above can be solved efficiently.  Instead of presenting algorithms for
solving each of the projection problems, in this section,  we  focus on designing algorithms for
computing the root of $\phi(\lambda)=0$.
%We are well aware that the algorithms proposed in this section can be easily extended to other problems mentioned above.
Once we have the root of $\phi(\lambda)=0$, we can easily obtain the  optimal solution.

\subsection{Breakpoint Search with Sorting}

In this subsection, we focus on designing algorithms for solving the nonsmooth equation  $\phi(\lambda)=0$
by assuming   $\bv$ is already sorted in descending order.
% After sorting, we have $s_i  \le \|\tilde\bv_{i+1}\|_1$, $i\in\{1,...,n-1\}$.
From Proposition \ref{lem.psi}, we know $\phi$ is piecewisely quadratic with breakpoints $\lambda_1, ..., \lambda_k$.
Therefore, $\phi$ is  quadratic on each interval $[\lambda_{j-1}, \lambda_j]$, $j=2, \ldots, k$.  If we know the index $j^*$ with  $\phi(v_{j^* +1})\ge 0$ and $\phi(v_{j^*})<0$, then solving $\phi=0$ on $(a, b)$ reverts to solving a quadratic equation with $\lambda$ on   $[v_{j^*+1}, v_{j^*   })$, which possesses  a close form of the root.  This motivates us to find such breakpoints $v_{j^*+1}$ and $v_{j^*}$, or equivalently, the index $j^*$.  Once this interval is determined,   from  \eqref{eq:QuadRoot}, we know $\phi(\lambda)$ has a  root $\lambda^*$ on $(v_{j^*+1}, v_{j^*})$:
\begin{equation}\label{rootPsi}
  \lambda^* = \frac{1}{j^*}\left(s  - t\sqrt{\tfrac{j^*w - s^2}{j^*-t^2}}\right).
\end{equation}

Notice that  the projection is often  sparse in the  applications of interests.
We propose a Forward Searching (FS) method to search for $j^*$, which  checks the function value at each breakpoint in the order of $v_1, \ldots, v_n$ to determine the first index $j$ satisfying  $\phi(v_{j+1}) \ge 0$.  One may also consider other searching strategies, e.g.,
backward searching in the order of $v_n, \ldots, v_1$.

The major computational cost besides sorting in this method is spent on the function evaluations of $\phi$, which takes  $O(n)$ operations per each    evaluation. Therefore, it can further reduce the computational cost by updating the function values  in the recursive way in our FS method. A complete statement of this method is provided in Algorithm~\ref{sub.alg.fs}.

\begin{algorithm}[t]
\caption{Forward searching (FS) method for solving $\phi(\lambda)=0$.} \label{sub.alg.fs}
\begin{algorithmic}[1]
\STATE Input a vector $\bv\in \R^n$  and  $t \in (1, \sqrt{n})$
\STATE Sort $\bv$ in decending order.
\STATE Let $j_t=\lceil t^2\rceil$ and compute $s = \sum_{i=1}^{j_t}v_i$, $w =\sum_{i=1}^{j_t}v_i^2$.
\FOR{$j=j_t,...,n-1$}
    \STATE $s  = s  + v_{j+1}, w   =  w + v_{j+1}^2$;
    \IF{$v_{j+1}<v_j$}
        \STATE  $\phi =\  ((j + 1) - t^2)((j+1)v_{j+1}-2s v_{j+1})v_{j+1} + s^2-t^2 w$;
        \IF{$\phi \ge 0 $}
            \STATE $ j^*  = j$; $s  = s  - v_{j+1}, w   =  w - v_{j+1}^2$;
            \RETURN
        \ENDIF
    \ENDIF
\ENDFOR
\STATE Output $\lambda^* = \frac{1}{j^*} \left(s  - t\sqrt{\frac{j^* w-s^2}{\hat j-t^2}} \right)$.
\end{algorithmic}
\end{algorithm}

%Obviously, the FS method terminates finitely, and the computational cost includes the sorting procedure and
%  the function evaluations  in a recursive manner for every index.

%  We summarize this result in the following theorem with the proof omitted.
%
%\begin{theorem}
%Algorithm~\ref{sub.alg.fs} terminates in finite iterations, and takes $O( n\log n + n)$ operations.
%\end{theorem}
%

%\subsection{Breakpoint Search without Sorting}
\subsection{Improved Bisection Methods}

The FS method  should be witnessed   complexity of $O(n\log n)$ at best   due to the presence of sorting.  In this subsection, we design a breakpoint search method without sorting, so that in practice  faster speed could be witnessed.

The efficiency of the our iterative root-finding algorithms mainly depends on two factors: the computational cost of
the function evaluation, and the total number of iterations. Next we discuss techniques to reduce the computational
efforts caused by these two factors.

%
%We note that, the two key factors that influence the efficiency of the root finding algorithm are: (1) the cost for evaluating the
%function value, and (2) the number of iterations. In what follows, we detail how to reduce the cost for evaluating $\phi(\lambda)$ in Subsection \ref{sec:BEval} and reduce the number of iterations in Subsection \ref{sec:BNumIter}.
%

%\subsection{Efficient evaluation of $\phi$}\label{sec:BEval}

For efficiently evaluating $\phi$,   \eqref{eq:psi} and \eqref{eq:Dpsi} imply that  the calculation of  $\phi(\lambda)$ and the first-order derivative information
depends on the calculation of $I_\lambda, s_\lambda$ and $q_\lambda$.
Notice that the first-order derivative of $\phi$ at $\lambda$ could be the gradient if $\phi$ is differentiable at $\lambda$ or
a subgradient  of $\phi$ at $\lambda$ if $\phi$ is nondifferentiable, since $\phi$ is convex on $(-\infty, \lambda_{j_t}]$ by Proposition \ref{lem.psi} (v).

%
%
%, which is the derivative of $\phi$ if $\lambda$ is not a break point or a subgradient of $\phi$ if $\lambda$ is any break point,  lie in the efficient calculation of $I_\lambda, s_\lambda$ and $q_\lambda$.

Let the  interval that we are working on  be $[l, r]$ with $\phi(l)>0$ and $\phi(r)<0$, and $\phi(r)$ and $\phi'(r)$ have been computed. The set $\Ical_r$ and the respective scales $m_r, s_r, q_r$ have been recorded. We now show how to evaluate the set $\Ical_\lambda$ and % the respective scales
 $I_\lambda, s_\lambda, q_\lambda$ for any $\lambda\in [l,r)$. Denote $\Ucal_\lambda=\{i: \lambda\le v_i\le r, i=1,\ldots,n\}$. Then we have
$\Ical_\lambda=\Ical_r\cup \Ucal_\lambda$, so that
\begin{equation}\label{eq:updatepsi}
 I_\lambda = m_r + |\Ucal_\lambda|,\ s_\lambda = s_r + \sum_{i\in \Ucal_\lambda}v_i,\  w_\lambda = w_r + \sum_{i\in \Ucal_\lambda}v_i^2,
\end{equation}
implying  that we can simply focus on those elements of $\bv$  on $[\lambda, r]$  for evaluation $\phi$ and the first-order derivative of  $\phi$. Note that the number of elements in $\bv$ in the interval $[\lambda, r]$ decreases as the algorithm proceeds,  thus the computational cost  needed per iteration decreases as well.

%\subsection{Bounding the Root}

{\bf Bounding the Root}. While solving an equation,
it could be  extremely helpful if one can find a relatively  accurate estimate of the range (interval) containing the root. This can significantly reduce the
number of iterations for many algorithms such as bisection method, secant method or even (nonsmooth) Newton method.
It has been shown that $\phi(\lambda)=0$ must have a unique root in $(0,v_{\max})$ if $\phi(0)>0$. The next proposition further narrows down the interval containing $\lambda^*$.

 \begin{proposition}\label{prop.bound root}
 For $\bv\in\mathbb{R}^n_+$,  suppose
 $
\|\bv\|_1>t, \
\|\bv\|_1 > t\|\bv\|_2  \  \text{ and }\  \|(\bv-\hat\lambda\mathbf{1})^+\|_2 > 1,
$
with $\hat\lambda$ being the root of $\psi(\lambda)=0$ on $\R$, $\lambda^*$ be the root $\phi(\lambda)=0$ on $(0, v_{\max})$, and
$\tilde \lambda: =(\|\bv\|_1-t\|\bv\|_2)/n > 0$.  Then $\phi(\tilde\lambda) >0$ and $\phi(\hat\lambda)<0$.
\end{proposition}

\begin{proof}
 First of all, by Theorem \ref{thm.nonnega}, we have $\phi(\hat\lambda)<0$ and $\hat\lambda\in (0, v_{\max})$.

Since $\phi$ is strictly decreasing on $(0,\lambda^*]$ by Proposition \ref{lem.psi} (iv), we only have to prove that $\tilde\lambda<\lambda^*$ since $\phi(\lambda^*)=0$.
By the definition of $\psi$ it holds true that
\begin{equation}\label{phi in lam}
\psi(\tilde\lambda) \ge \sum_{i=1}^n(v_i-\tilde\lambda)-t = t(\|\bv\|_2 - 1)
\end{equation}
where the equality comes from the definition of $\tilde\lambda$. Moreover, by $\hat\lambda>0$ and the definition of $\lambda^*$, $\psi$ and $\phi$,
 \[ \psi(\lambda^*) =\sqrt{ \phi(\lambda^*)+ t^2 \|(\bv-\hat\lambda\bm 1)^+\|_2^2 } - t = t\left(\|(\bv-\hat\lambda\bm 1)^+\|_2 -1 \right)  < t(\|\bv\|_2- 1),\]
 which, combined with \eqref{phi in lam}, yields $\psi(\tilde\lambda) > \psi(\lambda^*)$.  It follows that $\tilde\lambda < \lambda^*$ since $\psi$ is decreasing on $(0,\hat\lambda)$ by Proposition \ref{lem.phi}, completing the proof.
\end{proof}

 Proposition \ref{prop.bound root}  can be used to initialize Algorithm \ref{sub.alg.no sort1} and \ref{sub.alg.no sort2}.
 % In the PILL framework Algorithm~\ref{alg.basic}, $\hat\lambda$ is computed before solving $\phi(\lambda)=0$.
 In Algorithm \ref{sub.alg.no sort1} and \ref{sub.alg.no sort2} where $\bv$ is not sorted, we can set $l= \tilde\lambda\ \text{ and }\  r  = \hat\lambda$
 to further alleviate the computational effort.

%\subsection{Reducing the number of iterations}\label{sec:BNumIter}
{\bf Reducing the number of iterations.} Suppose we are currently working on two endpoints $l$ and $r$ with
$\phi(l)>0$ and $\phi(r)<0$.  We first derive a tighter upper bound for the root than $r$.
Consider the line passing through   points $(l, \phi(l))$ and $(r, \phi(r))$:
\[S(\lambda) = \phi(r) + \frac{\phi(l)-\phi(r)}{l-r}(\lambda-r)  \ \text{ with root }\
\lambda_S= r-\phi(r)\frac{l-r}{\phi(l)-\phi(r)}.
\]
Notice that $S(l)=\phi(l)>0$ and $S(r)=\phi(r)<0$, implying
  $l<\lambda_S<r$. % since $\phi(r)\frac{r-l}{\phi(r)-\phi(l)}>0$.
  Therefore, if we can show
$\phi(\lambda_S)<0$, then $\lambda_S$ is a tighter upper bound for the root than $r$.

If $r\le \lambda_{j_t}$, by Proposition \ref{lem.psi} (iv),
$\phi$ is convex on $(l,r)$.  Since
$S(l)=\phi(l)$ and $S(r)=\phi(r)$, the line $S(\lambda)$ is above the figure of $\phi$ on $[l,r]$, then it holds that $\phi(\lambda_S)<S(\lambda_S)= 0$.

If $   r> \lambda_{j_t}$, consider the line $\bar S(\lambda)$ passing through  points $(l, \phi(l))$ and $(\lambda_{j_t}, \phi(\lambda_{j_t}))$
\[\bar S(\lambda) = \phi(\lambda_{j_t}) + \frac{\phi(\lambda_{l})-\phi(j_t)}{l-\lambda_{j_t}}(\lambda-\lambda_{j_t})\]
 and denote the
 root of this line
% the intersection of this line with the lambda axis with
as $\bar \lambda$. Notice that $\phi(\lambda_{j_t})<0$ by the definition of $j_t$. Following the same argument
for $S(\lambda)$, it holds that $\phi(\bar\lambda)<0$. On the other hand, $l-r < l -\lambda_{j_t}$ and
$\phi(l) - \phi(r) < \phi(l)-\phi(\lambda_{j_t})$ since $\phi$ is monotonically increasing on $[\lambda_{j_t}, v_{\max})$ by Proposition \ref{lem.psi} (iv).
Therefore,  $\bar S'(\lambda) < S'(\lambda) < 0$, combined with  $S(l)=\bar S(l) > 0$, implying $\lambda_S > \bar\lambda$.
Hence, $\phi(\lambda_S)<0$ by Proposition \ref{lem.psi} (iv).
%
%Overall, we have shown  $\lambda_S$ forms a tighter upper bound of the uncertainty than $r$.

%On the other hand, $\phi(\lambda_{j_t})<\phi(r)$ since $\phi$ is monotonically decreasing by Proposition \ref{lem.psi}(iii). So we can see  that $\lambda_S>\bar\lambda$. Then

Next, we derive a tighter lower bound for the root than $l$.
Consider the line passing through $(l, \phi(l))$ with derivative $\varphi'_j(l)$, which by Proposition \ref{lem.psi} (v) can be written as
\[  T(\lambda) = \phi(l) + \phi'(l)(\lambda - l).\]
Based on the definition of $\varphi'_j$ in \eqref{eq:Dpsi} and $j\ge j_t$ by $\phi(l)>0$, we have $\phi'(l)=\varphi'_j(l)<0$. Therefore the unique root of $T$ is
\begin{equation}\label{eq:Tangent}
\lambda_T = l- \phi(l)/\phi'(l) > l.
\end{equation}
We only have to show that $\phi(\lambda_T)>0$. To see this, notice that $\phi$ is convex on $(-\infty, \lambda_{j_t})$ and $\varphi'_j(l) \in \partial \phi(l)$ by Proposition \ref{lem.psi} (v), the line $T(\lambda)$
always underestimate $\phi$, implying $\phi(\lambda_T)> T(\lambda_T)=0$.

% It follows that $\phi(\lambda_T)\ge 0$ and $\lambda_T>l$. Thus, $\lambda_T$ forms a tighter lower bound of the uncertainty than $l$.

We finally consider $\phi(\cdot)$ on the interval $(\lambda_{j+1}, \lambda_j)$ containing $l$.
 % Let $l\in (\lambda_{j+1}, \lambda_j)$.
  It follows that on this interval $\phi$ is equal to the quadratic function
\[\varphi_j(\lambda) = (m_l-t^2)m_l\lambda^2 -2(m_l-t^2)s_l\lambda + s_l^2 - t^2w_l.\]
Here $j\ge j_t$ since $\phi(l)> 0$. Denote the smaller root of $\varphi_j$ as $\lambda_Q$, then by  \eqref{eq:QuadRoot},
\begin{equation}\label{eq:Quad}
\lambda_Q = \frac{1}{m_l}\left( s_l  - t\sqrt{\frac{m_l w_l-s_l^2}{m_l-t^2}} \right),
\end{equation}
which is shown $\phi(\lambda_Q)\ge 0$ in the following lemma.
\begin{proposition}\label{prop.QuadApprax}
If  no $v_i\in (l, \lambda_Q)$, then $\phi(\lambda_Q)=0$; otherwise $\phi(\lambda_Q)>0$.
\end{proposition}

\begin{proof}
 If there no $v_i\in (l, \lambda_Q)$, then $\lambda_Q\in (\lambda_{j+1}, \lambda_j)$. So $\phi(\lambda_Q)=\varphi_j(\lambda_Q)=0$. Otherwise, Since $j\ge j_t$, so by Proposition \ref{lem.psi} (v), it holds that $\phi(\lambda_Q)>\varphi_j(\lambda_Q)=0$.
\end{proof}

Based on the discussion above, we design two methods.  One is named as the
Semi-Smooth Newton Secant Bisection method (SSNSB). The other is Quadratic Approximation Secant Bisection method (QASB). They are stated in Algorithm \ref{sub.alg.no sort1} and \ref{sub.alg.no sort2} respectivly.

\begin{algorithm}[t]
\caption{Semi-Smooth Newton Secant Bisection (SSNSB) Method for $\phi(\lambda)=0$.} \label{sub.alg.no sort1}
\begin{algorithmic}[1]
\STATE Input a vector $\bv\in \R^n$ and $t \in (1, \sqrt{n})$ and tolerance $\delta$.
\STATE Given $r$ and $l$ satisfying $\phi(l)>0$ and $\phi(r)<0$.

\WHILE{ $r-l > \delta$ and $|\phi(\lambda)|>\delta$}
    \STATE Compute $\lambda_S$,  $\phi(\lambda_S)$,  $\lambda_T$, $\phi(\lambda_T)$, $\lambda = (\lambda_S+\lambda_T)/2$, $\phi(\lambda)$.
           \IF{$\phi(\lambda)>0$}
            \STATE $r=\lambda_S$, $l=\lambda_T$.
              \ELSE
            \STATE $l=\lambda$, $r=\lambda_T$.
            \ENDIF
 \ENDWHILE
\STATE Output $\lambda$.
\end{algorithmic}
\end{algorithm}

\begin{algorithm}[t]
\caption{Quadratic Approximation Secant Bisection (QASB) Method for $\phi(\lambda)=0$.} \label{sub.alg.no sort2}
\begin{algorithmic}[1]
\STATE Input a vector $\bv\in \R^n$ and $t \in (1, \sqrt{n})$ and tolerance $\delta$.
\STATE Given  $r$ and $l$ satisfying $\phi(l)>0$ and $\phi(r)<0$.

\WHILE{ $r-l > \delta$ and $|\phi(\lambda)| > \delta$}
    \STATE Compute $\lambda_S$, $\phi(\lambda_S)$, $\lambda_Q$, $\phi(\lambda_Q)$, $\lambda = (\lambda_S+\lambda_Q)/2$, $\phi(\lambda)$.
        \IF{there is no $v_i$ in $(l, \lambda_Q)$}
            \STATE $\lambda =\lambda_Q$.
            \RETURN
        \ENDIF
         \IF{$\phi(\lambda)>0$}
            \STATE $r=\lambda_S$, $l=\lambda_Q$.
         \ELSE
            \STATE $l=\lambda$, $r=\lambda_Q$.
         \ENDIF
\ENDWHILE
\STATE Output $\lambda$.
\end{algorithmic}
\end{algorithm}

\section{Numerical experiments}\label{sec.experiment}

In this section, we test the performance of our proposed methods. The experiments consist of two parts.
The first part is testing  SSNSB and QASB with contemporary methods for computing the projection of a given $\bv$
onto $\B_1^{\text t}\cap\B_2$.  The main computational efforts focus on solving $\phi(\lambda)=0$,
corresponding to the case where $\bv$ lies in region $C_{\text{III}}$ in Table 1 (the hard case).
The second part is to test  SSNSB and QASB on  solving $\phi(\lambda)=0$ for computing the  projection onto $\mathbb{S}_1^{\text t}\cap\mathbb{S}_2$.
% , the purpose is to compare  with existing methods.

% The first part is solving $\phi(\lambda)=0$ arising in projection onto $\B_1^{\text t}\cap\B_2$, where $\bv$ lies in region $C_{\text{III}}$ in Table 1. The purpose is to compare QASB and other similar iterative algorithms.
%The second part is solving $\phi(\lambda)=0$ arising in projection onto $\mathbb{S}_1^{\text t}\cap\mathbb{S}_2$, the purpose is to compare QASB with existing methods. In all experiments, we take $\delta=1e$-$9$ as the tolerance error.
%

 In all experiments,  the tolerance is set as $\delta=10^{-9}$.
The codes are implemented  in C code, and runs on an HP laptop with a 1.80GHz Intel Core, i7-8565U CPU and 16.0GB RAM for Windows 10.
In the table of results,  each number is the average of 100 runs with fastest being the numbers in bold   and the standard deviation in the parenthesis.

\subsection{Projection onto $\B_1^{\text t}\cap\B_2$}\label{sec.bbtest}
%Given any vector $\bm{v}$, find the closest (in the euclidean sense) vector $\bm{x}$ in the intersection of an $\ell_1$ ball with radius $t$ and a unit $\ell_2$ ball.

For fixed dimension $n$, we randomly generate the input vector $\bm{v}\in\R^n$, and keep those that  fall
in $C_{\text{III}}$ and  discard those that are not in $C_{\text{III}}$. We only  compare
  the efficiency of FS, BM (the traditional Bisection method), SSNSB and QASB.
Let $\sigma=(\sqrt{n}-t)/(\sqrt{n}-1)$, where $\sigma$ measures sparseness as suggested  by Hoyer \cite{Hoyer04}. We set $\sigma$ for 0.9 as
used by \cite{Thom15}, and
setting $t$ according to the values of  $\sigma$ and $n$.
We consider the following
 three types of  approaches to generate the data, which is motivated  in the test of \cite{Dimitris16} and
  represent the components in $\bv$ have  different mean values.

%
%
%The purpose of using the three approaches is to obtain data from different distribution patterns to compare these methods more comprehensively.
%Note that components of $\bm{v}$ here is randomly ordered.
%for Type I, the data has one cluster point; for Type II, the data has two cluster points; for Type III, the data has three cluster points.

\begin{itemize}
\item

Type I:    $v_1,\cdots,v_n$ are i.i.d random Guassian numbers with mean $0$ and standard deviation  1.

\item

Type II:  7/8 components of $\bm{v}$ are i.i.d random Guassian numbers of mean $0$ and standard deviation  $0.2$,  the rest are i.i.d random Guassian numbers of mean $0.9$ and standard deviation $0.2$.

\item

Type III:   Components of $\bm{v}$ are randomly equally partitioned with i.i.d random Guassian numbers of mean $0.1$, $0.4$, $0.7$, $1.0$ respectively and standard deviation $0.2$.

\end{itemize}
%
%
%Type I:    $v_1,\cdots,v_n$ are i.i.d random Guassian numbers of mean $0$ and SD 1.
%
%Type II:  7/8 components of $\bm{v}$ are i.i.d random Guassian numbers of mean $0$ and SD $0.2$,  the rest are i.i.d random Guassian numbers of mean $0.9$ and SD $0.2$.
%
%Type III: Four quarters of components of $\bm{v}$ are i.i.d random Guassian numbers of mean $0.1$, $0.4$, $0.7$, $1.0$ respectively and SD $0.1$.

%${\text T_4}$: $v_1,\cdots,v_n$ are i.i.d random Uniform numbers of mean $0$ and SD 1.

The initial interval for SSNSB and QASB are chosen as ($\tilde\lambda$, $\hat\lambda$). As for BM,
we test initial interval with ($0$, $v_{\max}$)and ($\tilde\lambda$, $\hat\lambda$)  to show the benefits brought by  our proposed initial interval.
%(0, $\hat\lambda$) is used as the initial interval in BM and SSNSB in the experiment, and in QASB, ($\tilde\lambda$, $\hat\lambda$) is used as the initial interval as an improvement.

\textbf{An Illustrative Example}
We generate  $\bm{v}$ of Type I with $n=100$ to illustrate the behavior of QASB.
Table \ref{An Illustrative Example} shows the iteration information of BM, SSNSB and QASB over $k$.
We also record  $|\Ucal|$ of those algorithms during the iteration  in Table \ref{An Illustrative Example_2}, where
 Iter. denotes the total number of iterations.

 From  Table \ref{An Illustrative Example_2}, we can see that QASB converges faster than others.
Once $\left|\Ucal^k\right|=0$, QASB terminates, while other methods still need many iterations to terminate.
Overall,
 SSNSB and QASB outperform BM in number of iterations. % That is, the computational cost (related to $|U^k|$) for evaluating $\phi(.)$ decreases rapidly.
We can also see that our proposed initial interval is significantly narrowed.

\begin{table}
	\centering
	\caption{The iteration information of QASB terminating in 3 iterations.}
	\begin{tabular}{ccccccc}
		\hline
		$k$ & $|\Ucal^k|$&$l^k$&$\lambda_Q^k$&$\lambda^{k+1}$&$\lambda_S^k$&$r^k$\\ \hline
		0 & 50&0&0.5747&1.4858&1.9600&1.9610\\
		1 & 6&1.4858&1.8180&1.8846&1.9512&1.9600\\
		2 & 1&1.8846&1.9271&1.9300&1.9329&1.9512\\
		3 & 0&1.9271&-&1.9292&-&1.9300\\ \hline
	\end{tabular}
	\label{An Illustrative Example}
\end{table}

\begin{table}
	\centering
	\caption{$|\Ucal|$ of three algorithms on an example of Type I with $n=100$.
	We use our proposed initial interval  $[\tilde\lambda,\hat\lambda]$ for BM, SSNSB and QASB.
	We also test BM using this  interval versus using $[0, v_{\max}]$.
	 }
	\begin{tabular}{ccccc}
		\hline
		$k$ &BM $[0, v_{\max}]$ & BM $[\tilde\lambda, \hat \lambda ]$ & SSNSB&QASB\\
		\hline
		0 &100 &50 &50&50\\
		1 &15 &10 &8&6\\
		2 &12 &6 &1&1\\
		3 &7 &2 &0 &0\\
		4 &5 &2 &0 &\\
		5 &5 &1 &0 &\\
		6 &1 &0 &0 &\\
		7 &0 &0 &  &\\
		$\vdots$ & $\vdots$ & $\vdots$ &   &    \\
        31&0 &0 &  & \\
        32& 0 &  & & \\	\hline
	  Iter. &32 &31 &6&3\\ \hline
	\end{tabular}
	\label{An Illustrative Example_2}
\end{table}

\textbf{Number of Iterations }  We also test the total number of iterations needed by each algorithm for $\bv$ with increasing dimension
 $n=10^3,10^5,10^7$.
 The initial interval for all algorithms are set as $[\tilde \lambda, \hat\lambda]$.
 The results are summarized in Table \ref{steps}.

%\begin{figure}[t]
%	\centering
%	{\includegraphics[scale=0.35]{step_1.eps}}
%	{\includegraphics[scale=0.35]{step_2.eps}}
%	{\includegraphics[scale=0.35]{step_3.eps}}
%	\caption{Iteration steps for solving zero point of $\phi$, with three types of data, averaged over 100 realizations.
%	}
%	\label{fig:steps}
%\end{figure}

\begin{table}[t]
	\centering
	\caption{Iteration steps for solving zero point of $\phi$.}
	\begin{tabular}{cccccccccc}
		\hline
		%		&&$n=10(s\approx3.0e+0)$&$n=10^3 (s\approx30)$&$n=10^5 (s\approx1.9e3)$&$n=10^7 (s\approx1.8e5)$\\
		&\multicolumn{3}{c}{Type I} & \multicolumn{3}{c}{Type II} & \multicolumn{3}{c}{Type III}\\
		%\hline
		$n$&BM&SSNSB&QASB&BM&SSNSB&QASB&BM&SSNSB&QASB\\
		\hline
	$10^3$ & 31.0 & 6.1 & \textbf{4.0} & 30.0 & 6.4 & \textbf{3.8} & 30.0  & 6.6 & \textbf{4.0}\\
		%\hline
	$10^5$ & 32.0 & 6.9 & \textbf{4.6} & 30.0 & 6.2 & \textbf{5.4} & 30.0 & 3.6 & \textbf{5.3}\\
		%\hline
	$10^7$ & 32.0 & 7.0 & \textbf{6.0} & 31.0 & \textbf{6.4} & 6.5 & 30.0 & 6.3 &\textbf{6.1}\\
		\hline
	\end{tabular}
	\label{steps}
\end{table}

We can see from Table \ref{steps} that for  three     types of data, the total number of iterations needed by SSNSB and  QASB are all significantly less than BM, and  QASB needs usually two or three steps less than SSNSB.  Another aspect to notice is that the numbers of iterations needed by SSNSB and  QASB  remain stable for different
sizes of problems.

\textbf{Computation Efficiency}
The  CPU time (s) for each method is shown in Table \ref{Computation Efficiency1}, where we only record  computational time for  $\phi=0$ and
$N_z$ denotes  the number of non-zero elements of the obtained solution.

\begin{table}[t]
	\centering
	\caption{Computation time for for solving zero point of $\phi$.}
	\begin{tabular}{ccccc}
		\hline
%		&&$n=10(s\approx3.0e+0)$&$n=10^3 (s\approx30)$&$n=10^5 (s\approx1.9e3)$&$n=10^7 (s\approx1.8e5)$\\
		& $n$     &$10^3$        & $10^5$             &$10^7$\\
		 & $N_z$  &  3.0e+1 & 1.9e+3& 1.8e+5 \\
		\hline
		\multirow{4}{*}{Type I}
		&FS& 1.7e$-$4(8e$-$6) & 1.4e$-$2(2e$-$3) & 1.7e$+$0(2e$-$1)\\
		&BM& 6.2e$-$5(4e$-$6) & 3.6e$-$3(6e$-$4) & 3.9e$-$1(4e$-$2)\\
		&SSNSB& 2.9e$-$5(1e$-$6) &\textbf{1.3e$-$3}(3e$-$4) &\textbf{1.9e$-$1}(2e$-$2) \\
		&QASB& \textbf{2.8e$-$5}(1e$-$6) &\textbf{1.3e$-$3}(2e$-$4) &\textbf{1.9e$-$1}(2e$-$2)\\
		\hline
		\multirow{4}{*}{Type II}
		&FS& 1.8e$-$4(2e$-$5) & 1.7e$-$2(4e$-$3) & 1.6e$+$0(1e$-$1)\\
		&BM& 6.7e$-$5(1e$-$5) & 4.0e$-$3(1e$-$3) & 3.8e$-$1(2e$-$2)\\
		&SSNSB& 3.0e$-$5(4e$-$6) & \textbf{1.5e$-$3}(4e$-$4) & 1.5e$-$1(1e$-$2)\\
		&QASB& \textbf{2.7e$-$5}(4e$-$6) & \textbf{1.5e$-$3}(4e$-$4) & \textbf{1.4e$-$1}(1e$-$2)\\
		\hline
		\multirow{4}{*}{Type III}
		&FS& 1.2e$-$4(4e$-$5)&1.5e$-$2(3e$-$3)&1.6e$+$0(6e$-$2)\\
		&SSNSB& 2.2e$-$5(8e$-$6)&\textbf{1.4e$-$3}(3e$-$4)&1.6e$-$1(2e$-$2)\\
		&QASB& \textbf{2.1e$-$5}(9e$-$6)&\textbf{1.4e$-$3}(2e$-$4)&\textbf{1.3e$-$1}(2e$-$2)\\
		\hline
	\end{tabular}
	\label{Computation Efficiency1}
\end{table}

 Table \ref{Computation Efficiency1} indicates that FS always takes the longest computational time,
 which is mainly due the sorting procedure in this method.  Therefore, it cannot achieved faster observed
 complexity than $O(n \log n)$.
 %This is mainly because it nee.
 In particular, for large-scale cases ($n=10^7$),
 QASB is substantially outperforms FS and BM.
 Compared with SSNSB, QASB has obvious advantage in the number of iterations, yet the
 computational time may not be  always superior.
 The reason behind this observation may be the computational cost per iteration needed by SSNSB is low (simply
 the evaluation of a univariate quadratic function) when $|\Ucal_\lambda|=0$.

\subsection{Projection onto $\mathbb{S}_1^{\text t}\cap\mathbb{S}_2$}

We also apply our methods  on computing the projection of $\bv$ onto $\mathbb{S}_1^{\text t}\cap\mathbb{S}_2$.
In particular, we only consider the situation  $\|\bm{v}\|_1>t\|\bm{v}\|_2$ for a given $t$, in which case the computational effort is spent on  solving $\phi(\lambda)=0$ on $(0, |v|_{\text max})$.
 We compare QASB with   existing algorithms under the situation with   initial interval $(0, |v|_{\text max})$ in the experiment, since it is
 the fastest method from our previous experiments.

% (We resampled in case $\|\bm{v}\|_1\leq t\|\bm{v}\|_2$.)

The contemporary  algorithms we compare with include:
\begin{itemize}
	\item Alternating Projection method (AP)\cite{Hoyer04}.  % , (Patrik O. Hoyer 2004)
	\item Newton method  (NM1)\cite{Thom15}: % , (Markus Thom et al. 2015)
	 NM1 uses Newton method to iterate to find the zero point of $\phi_1(\lambda)$. (If the Newton step yields an iterate outside of  $[l,r]$, then it proceeds with a bisection iterate $\frac{l+r}{2}$.)
	% It is worth noting that the process constantly checks whether the iteration point is in the segment where $\phi_1(.)$ symbol changes. If it is, NM1 stops iterating.
	\item Newton method (NM2)\cite{Thom15}: % , (Markus Thom et al. 2015)
	it solves   $\phi_2(\lambda)=0$, 	a similar method to NM1.
% Here, it uses auxiliary function $\phi_2(\lambda)=\frac{\left\|{\rm max}(\bm{v}-\lambda\bm{e},0)\right\|_1^2}{\left\|{\rm max}(\bm{v}-\lambda\bm{e},0)\right\|_2^2}-t^2$ instead of $\phi_1$. The relationship between $\phi_2$ and $\phi$ is $\phi_2(\lambda)=\frac{\phi(\lambda)}{{\left\|{\rm max}(\bm{v}-\lambda\bm{e},0)\right\|_2^2}}$.
	\item Forward Search method, (FS)\cite{Gloaguen17}: % , (Arnaud Gloaguen et al. 2017)
	A forward search method for solving $\phi_1(\lambda)=0$ with a sorting procedure.
%	 Firstly, the components of $\bm{v}$ are sorted in descending order, and then $\phi_1(.)$ is calculated in turn. when the symbol of $\phi_1(.)$ changes, the quadratic function of the segment containing zero point is obtained, zero point of the quadratic function in the segment is exactly the zero point of $\phi_1(.)$.
%	\item Quadratic function zero-Secant-Bisection method, (QASB):
%	Auxiliary function is $\phi(\lambda)$. QASB uses  $\phi(l)\phi(r)$
%	$<0$ as the termination condition, where $l$ and $r$ are the left and right endpoints of $U$ respectively.
\end{itemize}

%Termination condition of QASB is different from that of NM1 and NM2. In theory, the stop criterion in NM1 and NM2 finds a more timely signal of the termination, but required more calculations.

%\textbf{Number of Iterations}
%For $n=10,10^3,10^5,10^7$,  we depict average number of iterations of AP, NM1, NM2 and QASB over 100 runs  in  Figure \ref{fig:steps222}.
%
%\begin{figure}[h]
%	\centering
%	{\includegraphics[width=5.0cm]{step_4.eps}}
%	\\
%	{\includegraphics[width=5.0cm]{step_5.eps}}
%	\\
%	{\includegraphics[width=5.0cm]{step_6.eps}}
%	\caption{
%		Iteration steps for the solving zero point of $\phi$, with three types of data and $\sigma=0.9$, averaged over 100 realizations.
%	}
%	
%	\label{fig:steps222}
%\end{figure}
%
%Through these experiments, we can see that the average number of iterations of AP is higher than the NW1, NW2 and QASB.
%For different types of data, iteration steps of the NW1, NW2 and QASB are different, and they are all within 8 steps.

% \textbf{Computation Efficiency}
We use the same data types as used in \S~\ref{sec.bbtest} and  compare the CPU time of these four algorithms in Table \ref{Computation Efficiency21} for $n=10^3,10^5,10^7$.
\begin{table}[t]
	\centering
			\caption{
		Comparison of computation times.  Each result is the average of  100 realizations with the standard deviations in the parenthesis. 	
	}
	\label{Computation Efficiency21}
	\begin{tabular}{ccccc} \hline
		&&$n=10^3$&$n=10^5$&$n=10^7$\\ \hline
		\multirow{5}{*}{Type I}
		&FS&7.4e$-$5(3e$-$5)&1.2e$-$2(4e$-$4)&5.2e$+$1(5e$+$0)\\
		&AP&5.7e$-$5(3e$-$5)&6.0e$-$3(3e$-$4)&1.5e+0(6e$-$1)\\
		&NM1&\textbf{1.5e$-$5}(4e$-$5)&\textbf{1.0e$-$3(1e$-$4)}&\textbf{2.6e$-$1}(6e$-$2)\\
		&NM2&2.0e$-$5(4e$-$5)&1.5e$-$3(1e$-$4)&3.4e$-$1(4e$-$2)\\
		&QASB&1.8e$-$5(4e$-$5)&1.8e$-$3(1e$-$4)&3.5e$-$1(9e$-$2)\\
		\hline
		
		\multirow{5}{*}{ Type II}
		&FS&1.7e$-$4(2e$-$5)&1.9e$-$2(3e$-$3)&4.5e$+$1(2e+0)\\
		&AP&1.2e$-$4(2e$-$5)&8.5e$-$3(1e$-$3)&1.2e$+$0(3e$-$1)\\
		&NM1&\textbf{3.5e$-$5}(6e$-$6)&1.8e$-$3(2e$-$4)&2.6e$-$1(4e$-$2)\\
		&NM2&4.3e$-$5(7e$-$6)&2.3e$-$3(4e$-$4)&3.1e$-$1(4e$-$2)\\
		&QASB&\textbf{3.5e$-$5}(5e$-$6)&\textbf{1.7e$-$3}(3e$-$4)&\textbf{2.3e$-$1}(1e$-$2)\\
		\hline
		
		\multirow{5}{*}{ Type III}
		&FS&7.2e$-$5(2e$-$5)&1.2e$-$2(4e$-$3)&3.0e$+$1(1e$+$0)\\
		&AP&6.0e$-$5(2e$-$5)&6.1e$-$3(2e$-$3)&8.5e$-$1(1e$-$1)\\
		&NM1&\textbf{1.8e$-$5}(6e$-$6)&1.3e$-$3(5e$-$4)&1.8e$-$1(2e$-$2)\\
		&NM2&\textbf{1.8e$-$5}(5e$-$6)&1.3e$-$3(5e$-$4)&1.8e$-$1(2e$-$2)\\
		&QASB&\textbf{1.8e$-$5}(4e$-$6)&\textbf{1.1e$-$3}(5e$-$4)&\textbf{1.3e$-$1}(5e$-$3)\\
		\hline		
	\end{tabular}

\end{table}

From Table \ref{Computation Efficiency21}, we can see that
NM1 is the fastest for Type I data, and   QASB is the fastest  for Type II and Type III data.

\section{Conclusions}\label{sec.conclusion}

In this paper, we have proposed, analyzed, and tested a unified approach for computing the projections onto the intersections of
$\ell_1$ and $\ell_2$ balls or spheres.  Novelties of our work is the proposed auxiliary function along with
its properties for characterizing the optimal solutions and a unified approach  for computing the solutions. The proposed approach has
bisection and Newton method implementations that can work with/without sorting.
The worst-case complexity of the proposed methods without sorting are $O(n\log n)$ and
the complexity in practice is observed to be $O(n)$. Our numerical experiments have demonstrated
the efficiency of the proposed methods.

{\bf Acknowledgements}
This research was supported by National Natural Science Foundation of China under Grant
11822103.

 %\begin{acknowledgements}
%If you'd like to thank anyone, place your comments here
%and remove the percent signs.
%\end{acknowledgements}

% BibTeX users please use one of
%\bibliographystyle{spbasic}      % basic style, author-year citations
\bibliographystyle{spmpsci_unsrt}
\bibliography{reference}   % name your BibTeX data base

\end{document}